\documentclass[a4paper]{article}
\usepackage{amsfonts}
\usepackage{amsmath}
\usepackage{amssymb}
\usepackage{amstext}
\usepackage{graphicx}

\newtheorem{theorem}{Theorem}

\newtheorem{corollary}[theorem]{Corollary}

\newtheorem{lemma}[theorem]{Lemma}

\newtheorem{proposition}[theorem]{Proposition}
\newtheorem{remark}[theorem]{Remark}

\newenvironment{proof}[1][Proof]{\noindent\textbf{#1.} }{\ \rule{0.5em}{0.5em}}
\input{tcilatex}
\begin{document}

\title{Construction of KdV flow I\\
{\large -Tau-function via Weyl function-}}
\author{Shinichi Kotani}
\date{}
\maketitle

\begin{abstract}
Sato introduced the tau-function to describe solutions to a wide class of
completely integrable differential equations. Later Segal-Wilson represented
it in terms of the relevant integral operators on Hardy space of the unit
disc. This paper gives another representation of the tau-functions by the
Weyl functions for 1d Schr\"{o}dinger operators with real valued potentials,
which will make it possible to extend the class of initial data for the KdV
equation to more general one.
\end{abstract}

\section{Introduction\protect\footnotetext[1]{%
2010 \textit{Mathematics Subject Classification} \ Primary 35Q53, 37K10 \
Secondary 35B15}}

The discovery of spectral invariants by \cite{g} was a trigger of the
succeeding rapid development of the study of the KdV equation%
\begin{equation*}
\partial_{t}f=6f\partial_{x}f-\partial_{x}^{3}f\text{.}
\end{equation*}
Since then, most of the works have been done by using the scattering data
for decaying solutions, and the discriminant for periodic solutions. On the
other hand, the algebraic structure of the KdV equation was revealed by \cite%
{s} based on the results by \cite{l}, \cite{h}, and provided a unified
approach to a wide class of integrable systems. Since his argument was
algebraic, so obtained solutions were rational, multi-solitons and
algebro-geometric ones, and all these solutions were described by
tau-functions. It has been a problem to what extent this method is effective
to obtain general solutions to the KdV equation such as solutions starting
from almost periodic functions. \cite{s-w} considered a kind of closure of
Sato's framework to obtain a certain class of transcendental solutions.
However, their solutions still remain in a meromorphic class on the entire
complex plane $\mathbb{C}$. It should be noted that \cite{m1} proposed an
algorithm to construct solutions to the KdV equation, although it seems that
his method also has difficulty to go beyond the class investigated by \cite%
{s-w}.

Since Sato's theory gives a unified way to solve the KdV equation at least
algebraically, there is some hope to exceed the already existing frameworks.
In the following the outline of Sato's method developed by \cite{s-w} is
described. Let $q$ be a function (for the present $q$ can be complex valued)
on $\mathbb{R}$. The basic assumption on $q$ is that the associated Schr\"{o}%
dinger equation%
\begin{equation}
-\partial_{x}^{2}f+qf=-z^{2}f  \label{1}
\end{equation}
has a Baker-Akhiezer function $f$ for all $z\in\mathbb{C}$ satisfying\ $%
\left\vert z\right\vert >R$ for some $R>0$, where $f$ is called a \emph{%
Baker-Akhiezer function} if $f$ has an expression%
\begin{equation*}
f\left( x,z\right) =e^{-xz}\left( 1+\sum_{n=1}^{\infty}a_{n}(x)z^{-n}\right)
\end{equation*}
converging on $\left\{ \left\vert z\right\vert >R\right\} $. This condition
on $q$ is equivalent to the reflectionless property on $\left(
R^{2},\infty\right) $ if the potential is real valued, which will be seen
later. For $r>R $ let $W$ be the closure of the linear span of $\left\{
f\left( x,z\right) \right\} _{x\in\mathbb{R}}$ in the Hilbert space $%
H=L^{2}\left( \left\vert z\right\vert =r\right) $ ($\mathbb{D}_{r}=\left\{
z\in \mathbb{C}\text{; }\left\vert z\right\vert <r\right\} $, $\partial 
\mathbb{D}_{r}=\left\{ z\in\mathbb{C}\text{; }\left\vert z\right\vert
=r\right\} $). From (\ref{1}) we have easily\medskip\newline
(W.1) \ If $f\in W$, then $z^{2}f\in W$.\medskip\newline
Let $H_{+}=$\textrm{span}$\left\{ z^{n}\text{; }n=0,1,2,\cdots\right\} $, $%
H_{-}=$\textrm{span}$\left\{ z^{n}\text{; }n=-1,-2,\cdots\right\} $ in $H$,
and $\mathfrak{p}_{\pm}$ be the orthogonal projections to $H_{\pm}$
respectively. Define the second property by\medskip\newline
(W.2) \ $\mathfrak{p}_{+}$: $W\rightarrow H_{+}$ is bijective.\medskip%
\newline
The totality of closed subspaces of $H$ satisfying (W.1), (W.2) is denoted
by $Gr^{\left( 2\right) }$, which will be replaced by $Gr^{\left( 2\right)
}\left( \mathbb{D}_{r}\right) $ if it is necessary. In a general setting of $%
Gr^{\left( 2\right) }$ the property (W.2) is replaced by the Fredholm
condition of $\mathfrak{p}_{+}$. Set%
\begin{equation*}
\mathit{\Gamma}=\left\{ g=e^{h}\text{; }h\text{ is holomorphic on }\left\{
\left\vert z\right\vert <r^{\prime}\right\} \text{ for some }r^{\prime
}>r\right\} \text{.}
\end{equation*}
Then, for $W\in Gr^{\left( 2\right) }$, $g\in\mathit{\Gamma}$ a new subspace 
$gW$ satisfies (W.1), however the property (W.2) is not always valid. For a
given $W\in Gr^{\left( 2\right) }$ let $\delta>0$ be such that $e^{xz}W\in
Gr^{\left( 2\right) }$ holds for any $x\in\mathbb{R}$ such that $\left\vert
x\right\vert <\delta$. To find such a $\delta$ is always possible because $%
e^{xz}W\in Gr^{\left( 2\right) }$ for $x=0$. Then, an adhoc derivation of
the potential $q$ and the Baker-Akhiezer function from $W $ is as follows.

The property (W.2) for $e^{xz}W$ implies that there exists a unique element $%
f$ of $W$ satisfying%
\begin{equation*}
e^{xz}f\left( z\right) \in1+H_{-}\text{,}
\end{equation*}
which is denoted by $f(x,z)$. Since any element of $H_{-}$ has Taylor
expansion at $z=\infty$, we have 
\begin{equation*}
u\left( x,z\right) \equiv e^{xz}f\left( x,z\right) -1=\frac{a_{1}\left(
x\right) }{z}+\frac{a_{2}\left( x\right) }{z^{2}}+\frac{a_{3}\left( x\right) 
}{z^{3}}+\cdots\text{.}
\end{equation*}
Taking derivative with respect to $x$ yields%
\begin{equation*}
\left\{ 
\begin{array}{l}
f^{\prime}\left( x,z\right) =-ze^{-xz}\left( 1+u\left( x,z\right) \right)
+e^{-xz}u^{\prime}(x,z)\smallskip \\ 
f^{\prime\prime}\left( x,z\right) =z^{2}e^{-xz}\left( 1+u\left( x,z\right)
\right) -2ze^{-xz}u^{\prime}(x,z)+e^{-xz}u^{\prime\prime}(x,z)%
\end{array}
,\right.
\end{equation*}
which shows%
\begin{align*}
& e^{xz}\left( -f^{\prime\prime}\left( x,z\right) -2a_{1}^{\prime}\left(
x\right) f\left( x,z\right) +z^{2}f\left( x,z\right) \right) \\
& =2zu^{\prime}(x,z)-u^{\prime\prime}(x,z)-2a_{1}^{\prime}\left( x\right)
\left( 1+u\left( x,z\right) \right) \\
& =\sum_{k=1}^{\infty}\left( 2a_{k+1}^{\prime}\left( x\right)
-a_{k}^{\prime\prime}\left( x\right) -2a_{1}^{\prime}\left( x\right)
a_{k}\left( x\right) \right) z^{-k}\text{.}
\end{align*}
Since the first term belongs to $e^{xz}W$ due to (W,1) and the last term is
an element of $H_{-}$, we have%
\begin{equation}
\left\{ 
\begin{array}{l}
-f^{\prime\prime}\left( x,z\right) -2a_{1}^{\prime}\left( x\right) f\left(
x,z\right) +z^{2}f\left( x,z\right) =0\smallskip \\ 
2a_{k+1}^{\prime}\left( x\right) -a_{k}^{\prime\prime}\left( x\right)
-2a_{1}^{\prime}\left( x\right) a_{k}\left( x\right) =0,\text{ \ }%
k=1,2,\cdots%
\end{array}
\right. \text{,}  \label{2}
\end{equation}
where we have used the property (W.2) again for the space $e^{xz}W$.
Therefore, if we define%
\begin{equation}
q_{W}(x)=-2a_{1}^{\prime}\left( x\right) \text{,}  \label{3}
\end{equation}
then $f\left( x,z\right) $ is the Baker-Akhiezer function for $q_{W}$.

A solution to the KdV equation starting from $q_{W}$ is obtained similarly.
Assume $e^{xz+tz^{3}}W\in Gr^{\left( 2\right) }$ for any $x,t\in\mathbb{R}$.
Let $f\left( x,t,z\right) $ be a unique element $f$ of $W$ satisfying

\begin{equation*}
e^{xz+tz^{3}}f\left( z\right) \in1+H_{-}\text{,}
\end{equation*}
and set%
\begin{equation*}
\left\{ 
\begin{array}{l}
e^{xz+tz^{3}}f\left( x,t,z\right) =1+u\left( x,t,z\right) \text{ \ \ with \ }%
u\left( x,t,z\right) \in H_{-} \\ 
u\left( x,t,z\right) =\dfrac{a_{1}\left( x,t\right) }{z}+\dfrac {a_{2}\left(
x,t\right) }{z^{2}}+\dfrac{a_{3}\left( x,t\right) }{z^{3}}+\cdots \\ 
q_{W}\left( x,t\right) =a_{1}^{\prime}\left( x,t\right)%
\end{array}
\text{,}\right.
\end{equation*}
where $^{\prime}$ denotes the derivative with respect to $x$. Taking
derivatives of $e^{xz+tz^{3}}f$ with respect to $t,x$ yields%
\begin{equation*}
\left\{ 
\begin{array}{l}
e^{xz+tz^{3}}\partial_{t}f=-z^{3}\left( 1+u\right) +\partial_{t}u \\ 
e^{xz+tz^{3}}f^{\prime}=-z\left( 1+u\right) +u^{\prime} \\ 
e^{xz+tz^{3}}f^{\prime\prime\prime}=-z^{3}\left( 1+u\right)
+3z^{2}u^{\prime}-3zu^{\prime\prime}+u^{\prime\prime\prime}%
\end{array}
\text{,}\right.
\end{equation*}
and hence we obtain%
\begin{equation*}
e^{xz+tz^{3}}\left( \partial_{t}f-f^{\prime\prime\prime}-3a_{1}^{\prime
}f^{\prime}\right) =\partial_{t}u-3z^{2}u^{\prime}+3zu^{\prime\prime
}-u^{\prime\prime\prime}+3za_{1}^{\prime}\left( 1+u\right)
-3u^{\prime}a_{1}^{\prime}\text{.}
\end{equation*}
Then, the right side function is%
\begin{equation*}
-3a_{2}^{\prime}+3a_{1}^{\prime\prime}+3a_{1}^{\prime}a_{1}\text{ \ \ modulo
\ }H_{-}\text{.}
\end{equation*}
Applying the second identities of (\ref{2}) to $e^{tz^{3}}f\left(
x,t,z\right) $ and $e^{tz^{3}}W$ yields%
\begin{equation*}
-3a_{2}^{\prime}+3a_{1}^{\prime\prime}+3a_{1}^{\prime}a_{1}=\dfrac{3}{2}%
\left( -a_{1}^{\prime\prime}-2a_{1}^{\prime}a_{1}+2a_{1}^{\prime\prime
}+2a_{1}^{\prime}a_{1}\right) =\dfrac{3}{2}a_{1}^{\prime\prime}\text{,}
\end{equation*}
hence%
\begin{align*}
& e^{xz+tz^{3}}\left( \partial_{t}f-f^{\prime\prime\prime}-3a_{1}^{\prime
}f^{\prime}-\dfrac{3}{2}a_{1}^{\prime\prime}f\right) \\
& =\partial_{t}u-3z^{2}u^{\prime}+3zu^{\prime\prime}-u^{\prime\prime\prime
}+3za_{1}^{\prime}\left( 1+u\right) -3u^{\prime}a_{1}^{\prime}-\dfrac{3}{2}%
a_{1}^{\prime\prime}\left( 1+u\right) \text{,}
\end{align*}
which is an element of $H_{-}$. Therefore, the property (W.2) for $%
e^{xz+tz^{3}}W$ implies%
\begin{equation*}
\partial_{t}f-f^{\prime\prime\prime}-3a_{1}^{\prime}f^{\prime}-\dfrac{3}{2}%
a_{1}^{\prime\prime}f=0\text{,}
\end{equation*}
thus%
\begin{equation*}
\partial_{t}u-3z^{2}u^{\prime}+3zu^{\prime\prime}-u^{\prime\prime\prime
}+3za_{1}^{\prime}\left( 1+u\right) -3u^{\prime}a_{1}^{\prime}-\dfrac{3}{2}%
a_{1}^{\prime\prime}\left( 1+u\right) =0\text{.}
\end{equation*}
The coefficient of $z^{-1}$ of the above left side function is%
\begin{equation*}
\partial_{t}a_{1}-3a_{3}^{\prime}+3a_{2}^{\prime\prime}-a_{1}^{\prime
\prime\prime}+3a_{1}^{\prime}a_{2}-3\left( a_{1}^{\prime}\right) ^{2}-\dfrac{%
3}{2}a_{1}^{\prime\prime}a_{1}=0\text{.}
\end{equation*}
The second identity of (\ref{2}) for $e^{tz^{3}}W$ yields%
\begin{equation*}
2a_{2}^{\prime}-a_{1}^{\prime\prime}-2a_{1}^{\prime}a_{1}=0,\text{ \ \ }%
2a_{3}^{\prime}-a_{2}^{\prime\prime}-2a_{1}^{\prime}a_{2}=0\text{,}
\end{equation*}
hence%
\begin{align*}
&
-3a_{3}^{\prime}+3a_{2}^{\prime\prime}-a_{1}^{\prime\prime\prime}+3a_{1}^{%
\prime}a_{2}-3\left( a_{1}^{\prime}\right) ^{2}-\dfrac{3}{2}%
a_{1}^{\prime\prime}a_{1} \\
& =-\dfrac{3}{2}\left( a_{2}^{\prime\prime}+2a_{1}^{\prime}a_{2}\right)
+3a_{2}^{\prime\prime}-a_{1}^{\prime\prime\prime}+3a_{1}^{\prime}a_{2}-3%
\left( a_{1}^{\prime}\right) ^{2}-\dfrac{3}{2}a_{1}^{\prime\prime }a_{1} \\
& =-\dfrac{1}{4}a_{1}^{\prime\prime\prime}-\dfrac{3}{2}\left( a_{1}^{\prime
}\right) ^{2}
\end{align*}
Consequently, we have%
\begin{equation*}
\partial_{t}a_{1}-\dfrac{1}{4}a_{1}^{\prime\prime\prime}-\dfrac{3}{2}\left(
a_{1}^{\prime}\right) ^{2}=0\text{.}
\end{equation*}
Taking derivative with respect to $x$ shows%
\begin{equation*}
-\dfrac{1}{2}\partial_{t}q_{W}+\dfrac{1}{8}q_{W}^{\prime\prime\prime}-\dfrac{%
3}{4}q_{W}q_{W}^{\prime}=0\text{ \ hence \ }\partial_{t}q_{W}-\dfrac{1}{4}%
q_{W}^{\prime\prime\prime}+\dfrac{3}{2}q_{W}q_{W}^{\prime }=0\text{,}
\end{equation*}
which shows that $q_{W}(x,-4t)$ satisfies the KdV equation with initial data 
$q_{W}(x)$. The idea behind this calculation is an effective use of algebra
of pseudo differential operators, and a more systematic argument can be
found in \cite{s-w}.

The tau-function was introduced by Sato to describe the $\mathit{\Gamma}$%
-action on $Gr^{(2)}$. For $W\in Gr^{\left( 2\right) }$ and $f_{+}\in H_{+}$
let $f_{-}$ be the unique element of $H_{-}$ such that%
\begin{equation*}
f_{+}+f_{-}\in W
\end{equation*}
is valid, which is possible due to (W.2). An operator $A_{W}$ from $H_{+}$
to $H_{-}$ is defined by $A_{W}f_{+}=f_{-}$. Then, for $g\in\mathit{\Gamma} $
the \emph{tau-function} is defined by%
\begin{equation}
\tau_{W}(g)=\det\left( I+g^{-1}\mathfrak{p}_{+}gA_{W}\right) \text{,}
\label{4}
\end{equation}
and the functions $q_{W}(x)$, $q_{W}(x,t)$ are given by the tau-functions in
Lemma\ref{l4} as follows:%
\begin{equation*}
q_{W}(x)=-2\partial_{x}^{2}\log\tau_{W}(e^{xz})\text{, \ \ }%
q_{W}(x,t)=-2\partial_{x}^{2}\log\tau_{W}(e^{xz+tz^{3}})\text{.}
\end{equation*}
If we use $e^{xz+tz^{2n+1}}$ in place of $e^{xz+tz^{3}}$, then we obtain
solutions to the higher order KdV equations. It is known that $gW\in
Gr^{\left( 2\right) }$ if and only if $\tau_{W}(g)\neq0$, and $\tau_{W}(g) $
has a cocycle property%
\begin{equation*}
\tau_{W}(g_{1}g_{2})=\tau_{W}(g_{1})\tau_{g_{1}W}(g_{2})
\end{equation*}
(see Proposition\ref{p3}), which will leads us to the definition of the KdV
flow.

To treat real valued potentials some notions for $Gr^{(2)}$ are necessary.
For a function $f$ on a domain of $\mathbb{C}$ set $\overline{f}\left(
z\right) =\overline{f\left( \overline{z}\right) }$. Define $\overline{W}%
=\left\{ \overline{f}\text{ ; }f\in W\right\} $ and set%
\begin{equation*}
Gr_{\func{real}}^{(2)}=\left\{ W\in Gr^{(2)}\text{; }W=\overline {W}\right\} 
\text{, \ }\mathit{\Gamma}_{\func{real}}=\left\{ g\in\mathit{\Gamma}\text{; }%
g=\overline{g}\right\} \text{,}
\end{equation*}
and%
\begin{equation*}
Gr_{+}^{\left( 2\right) }=\left\{ W\in Gr_{\func{real}}^{\left( 2\right) }%
\text{; \ }\tau_{W}(g)\geq0\text{ for any }g\in\mathit{\Gamma }_{\func{real}%
}\right\} \text{.}
\end{equation*}
For $W\in Gr_{\func{real}}^{(2)}$ the corresponding potential $q_{W}$ takes
real values. The first theorem is

\begin{theorem}
\label{t2}An identity%
\begin{equation*}
Gr_{+}^{\left( 2\right) }=\left\{ W\in Gr_{\func{real}}^{\left( 2\right) }%
\text{; \ }gW\in Gr^{\left( 2\right) }\text{ for any }g\in\mathit{\Gamma}_{%
\func{real}}\right\}
\end{equation*}
holds.
\end{theorem}

A sufficient condition for $W$ to be an element of $Gr_{+}^{\left( 2\right)
} $ will be given by using the $m$-function. Although the potential $q_{W}$
is obtained from $W\in Gr^{\left( 2\right) }$ through $\tau_{W}$, this
correspondence is not one to one. The quantity determining $q_{W}$ is called
in this paper $m$\emph{-function} given by%
\begin{equation*}
m_{W}(z)=-\dfrac{f^{\prime}(0,z)}{f(0,z)}
\end{equation*}
with the Baker-Akhiezer function $f$. \S\ 5 is devoted to investigation of $%
m $-functions, especially $\tau_{W}$ is decomposed into two components, one
of which is expressed by $\tau_{m_{W}}$ and depends only on $m_{W}$.

Usually integrable Hamiltonian systems have been linearized through
action-angle variables. In order to apply this point of view to the KdV
equation we need at least integrability of the solutions, namely $%
\int_{-\infty}^{\infty}\left\vert u\left( x\right) \right\vert dx<\infty$ in
decaying case and $\int_{0}^{1}\left\vert u\left( x\right) \right\vert
dx<\infty$ in periodic case, and one can define conserved quantities
suitably. For almost periodic case \cite{j-m}, \cite{mo} considered the KdV
equation in the framework of Hamiltonian systems. However, the lack of
compactness and connectedness of the invariant leaves prohibits to develop
the argument further. Sato did not take this approach and constructed
directly the flow by the tau-functions. Therefore, in his theory the
tau-function is the crucial quantity. In the context of Sato's theory the
Weyl functions were first used by \cite{j1} to define an element of Sato's
Grassmann manifold. The purpose of the present paper is to give a
representation of the tau-functions by Weyl functions for 1d Schr\"{o}dinger
operators with real potentials, so that one can obtain more general
solutions.

To state the main results we need more terminologies. Suppose a Schr\"{o}%
dinger operator $L_{q}=-\partial_{x}^{2}+q$ with real valued $q\in
L_{loc}^{1}\left( \mathbb{R}\right) $ is essentially self-adjoint on $%
L^{2}\left( \mathbb{R}\right) $ (the boundedness of $q$ is sufficient for
this). Then it is known that $\dim\left\{ f\in L^{2}\left( \mathbb{R}_{\pm
}\right) \text{; }L_{q}f=zf\right\} =1$ for every $z\in\mathbb{C}\backslash%
\mathbb{R}$. The \emph{Weyl functions} $m_{\pm}$ are defined by%
\begin{equation*}
m_{\pm}\left( z\right) =\pm\dfrac{f_{\pm}^{\prime}\left( 0,z\right) }{%
f_{\pm}\left( 0,z\right) }
\end{equation*}
with two non-trivial $f_{\pm}\in\left\{ f\in L^{2}\left( \mathbb{R}_{\pm
}\right) \text{; }L_{q}f=zf\right\} $. $m_{\pm}$ are holomorphic on $\mathbb{%
C}\backslash\mathbb{R}$ and have positive imaginary parts on $\mathbb{C}_{+}$%
. Gelfand-Levitan, Marchenko showed that $q$ can be recovered from $m_{\pm}$
uniquely (see \cite{m1}). A potential $q$ is called \emph{reflectionless} on 
$F\in\mathcal{B}\left( \mathbb{R}\right) $ if its Weyl functions $m_{\pm}$
satisfy%
\begin{equation}
m_{+}\left( \xi+i0\right) =-\overline{m_{-}\left( \xi+i0\right) }\left(
=-m_{-}\left( \xi-i0\right) \right) \text{ \ \ a.e. }\xi\in F\text{.}
\label{5}
\end{equation}
Set%
\begin{equation}
m(z)=\left\{ 
\begin{tabular}{cc}
$-m_{+}\left( -z^{2}\right) $ & $\text{if \ }\func{Re}z>0$ \\ 
$m_{-}\left( -z^{2}\right) $ & $\text{if \ }\func{Re}z<0$%
\end{tabular}
\right. \text{,}  \label{13}
\end{equation}
and assume that there exist $\lambda_{0}<0<\lambda_{1}$ such that%
\begin{equation*}
\inf\text{\textrm{sp}}L_{q}>\lambda_{0}\text{, and }q\text{ is
reflectionless on }\left( \lambda_{1},\infty\right) \text{.}
\end{equation*}
Then, $m$ is holomorphic on $\mathbb{C}\backslash\left( \left[ -\sqrt{%
-\lambda_{0}},\sqrt{-\lambda_{0}}\right] \cup i\left[ -\sqrt {\lambda_{1}},%
\sqrt{\lambda_{1}}\right] \right) $, and has an expansion at $z=\infty$ like%
\begin{equation*}
m\left( z\right) =z+m_{1}z^{-1}+m_{2}z^{-2}+\cdots\text{.}
\end{equation*}

It will be seen later that this property of $m$ implies the existence of the
Baker-Akhiezer function with $r>\sqrt{\left( -\lambda_{0}\right) \vee
\lambda_{1}}$, and%
\begin{equation*}
W_{m}=\left\{ \varphi\left( z^{2}\right) +\psi\left( z^{2}\right) m(z)\text{%
; \ }\varphi,\text{ }\psi\in H_{+}\right\}
\end{equation*}
is an element of $Gr_{+}^{\left( 2\right) }$ (see Proposition\ref{p11}).

The second theorem is related to the construction of KdV flow by using $m$.
Set%
\begin{equation}
\left\{ 
\begin{array}{l}
\mathcal{Q}_{\infty}=\left\{ 
\begin{array}{c}
q\text{; }q\text{ is reflectionless on }\left( r^{2},\infty\right) \text{ and%
} \\ 
\text{\textrm{sp} }L_{q}\subset\lbrack-r^{2},\infty)\text{ for some }r>0%
\end{array}
\right\} \smallskip \\ 
\mathit{\Gamma}_{\func{real}}^{\infty}=\left\{ g=e^{h}\text{; \ \ }h\text{
is an entire function with }\overline{h}=h\text{.}\right\}%
\end{array}
\right. \text{,}  \label{59}
\end{equation}
and for $q\in\mathcal{Q}_{\infty}$ define%
\begin{equation*}
\left( K(g)q\right) \left( x\right)
=-2\partial_{x}^{2}\log\tau_{W_{m}}\left( ge_{x}\right) \text{,}
\end{equation*}
where $m$ is the $m$-function defined by (\ref{13}). Then, we have

\begin{theorem}
\label{t3}$\left\{ K(g)\right\} _{g\in\mathit{\Gamma}_{\func{real}%
}^{\infty}} $ defines a flow on $\mathcal{Q}_{\infty}$ such that%
\begin{equation*}
\left\{ 
\begin{array}{l}
\left( K\left( e^{tz}\right) q\right) \left( x\right) =q\left( x+t\right) 
\text{,\smallskip} \\ 
\left( K\left( e^{-4tz^{3}}\right) q\right) \left( x\right) \text{ \textit{%
satisfies the KdV equation}.}%
\end{array}
\right.
\end{equation*}
\end{theorem}

The next task is to represent the tau-function $\tau _{W_{m}}$ more
concretely. For a function $f$ set%
\begin{equation*}
f_{e}(z)=\dfrac{1}{2}\left( f\left( \sqrt{z}\right) +f\left( -\sqrt{z}%
\right) \right) \text{, \ \ }f_{o}(z)=\dfrac{1}{2\sqrt{z}}\left( f\left( 
\sqrt{z}\right) -f\left( -\sqrt{z}\right) \right) \text{.}
\end{equation*}%
Let $D$, $D^{\prime }$ be simply connected bounded domains in $\mathbb{C}$
containing the interval $\left[ -\lambda _{1},-\lambda _{0}\right] $, and
satisfy $D\subset D^{\prime }$. Set $C=\partial C$, $C^{\prime }=\partial
D^{\prime }$ the boundaries of $D$, $D^{\prime }$ respectively. We assume $C$%
, $C^{\prime }$ are smooth curves and surround $\left[ -\lambda
_{1},-\lambda _{0}\right] $ counterclockwise. For $\delta $ whose $\delta
_{e}$, $\delta _{o}$ are holomorphic in a simply connected domain including $%
C^{\prime }$, set $\widetilde{m}(z)=m(z)-\delta \left( z\right) $, and define%
\begin{equation}
\left\{ 
\begin{array}{l}
M_{g}\left( z,\lambda \right) =\dfrac{\widehat{g}_{o}\left( z\right) \left( g%
\widetilde{m}\right) _{e}\left( \lambda \right) +\widehat{g}_{e}\left(
z\right) \left( g\widetilde{m}\right) _{o}\left( \lambda \right) }{\lambda -z%
}\smallskip  \\ 
N_{g}\left( z,\lambda \right) =\dfrac{1}{2\pi i}\dint_{C^{\prime }}\dfrac{%
M_{g}\left( \lambda ^{\prime },\lambda \right) }{\lambda ^{\prime }-z}%
m_{o}\left( \lambda ^{\prime }\right) ^{-1}d\lambda ^{\prime }\smallskip  \\ 
\left( N_{m}(g)f\right) \left( z\right) =\dfrac{1}{2\pi i}%
\dint_{C}N_{g}\left( z,\lambda \right) f\left( \lambda \right) d\lambda 
\text{.}%
\end{array}%
\right.   \label{57}
\end{equation}%
where $\widehat{g}\left( z\right) =g(z)^{-1}$. The operator $N_{m}(g)$
defines a trace class operator on $L^{2}\left( C\right) $. \medskip 

\begin{theorem}
\label{t1}For $q\in\mathcal{Q}_{\infty}$ and $g\in\mathit{\Gamma }_{\func{%
real}}^{\infty}$ we have%
\begin{equation*}
\tau_{W_{m}}\left( g\right) =\det\left( I+N_{m}(g)\right) \text{.}
\end{equation*}
\end{theorem}

\begin{remark}
\textit{The class }$\mathcal{Q}_{\infty}$\textit{\ contains multi-solitons,
algebro-geometric solutions and they are dense in }$\mathcal{Q}_{\infty}$%
\textit{. Especially all solutions considered in \cite{g-k-z} are included
if }$\inf$\textrm{sp}$L_{q}>-\infty$\textit{. A necessary condition was
obtained by \cite{lu}, namely she} \textit{proved that} if $\lambda_{0}=\inf$
\textrm{sp}$L_{q}$, and $q$ is reflectionless on $\left(
\lambda_{1},\infty\right) $, then $q$\textit{\ is holomorphic on }$\left\{
\left\vert \func{Im}z\right\vert <\left( \lambda_{1}-\lambda_{0}\right)
^{-1/2}\right\} $\textit{\ and has bound}%
\begin{equation*}
\left\vert q(z)-\lambda_{1}\right\vert \leq2\left( \lambda_{1}-\lambda
_{0}\right) \left( 1-\sqrt{\lambda_{1}-\lambda_{0}}\left\vert \func{Im}%
z\right\vert \right) ^{-2}\text{.}
\end{equation*}
\textit{Since }$\tau_{W}\left( e_{x}\right) $\textit{\ is entire as a
function of }$x$,\textit{\ we know from Theorem\ref{t3} that }$q(x)$\textit{%
\ is meromorphic on the entire complex plane }$\mathbb{C}$.\medskip
\end{remark}

This paper is intended to be self-contained, so many results are overlapping
with those of \cite{s-w}. The author already published one paper \cite{k} on
the property $\tau_{W}(g)\neq0$. However, it contains several mistakes in
the proofs, moreover, the whole story was not well organized. The present
paper tries to improve these points by employing several basic notions from
the theory of Toeplitz operators.

Throughout the paper we use the following notations. $\mathbb{R}$ denotes
the real line and $\mathbb{C}$ denotes the whole complex plane, and%
\begin{equation*}
\left\{ 
\begin{array}{l}
\mathbb{R}_{+}=\left\{ x\in\mathbb{R}\text{; \ }x>0\right\} \text{, \ \ \ \
\ }\mathbb{R}_{-}=\left\{ x\in\mathbb{R}\text{; \ }x<0\right\} \\ 
\mathbb{C}_{+}=\left\{ z\in\mathbb{C}\text{; \ }\func{Im}z>0\right\} \text{,
\ }\mathbb{C}_{-}=\left\{ z\in\mathbb{C}\text{; \ }\func{Im}z<0\right\} \\ 
\mathbb{D}_{r}=\left\{ z\in\mathbb{C}\text{; \ }\left\vert z\right\vert
<r\right\} \text{, \ \ \ \ \ \ }\partial\mathbb{D}_{r}=\left\{ z\in\mathbb{C}%
\text{; \ }\left\vert z\right\vert =r\right\}%
\end{array}
\right. \text{.}
\end{equation*}
Moreover, $\sqrt{z}$ is defined as a univalent function on $\mathbb{C}%
\backslash\mathbb{R}_{-}$ so that $\sqrt{1}=1$, hence $\sqrt{z}$ satisfies $%
\sqrt{z}=\overline{\sqrt{\overline{z}}}$.

\section{Grassmann manifold $Gr^{\left( 2\right) }$}

For completeness sake we define the relevant spaces again. Let $%
H=L^{2}\left( \partial\mathbb{D}_{r}\right) $ and set%
\begin{equation*}
H_{+}=\text{the span of }\left\{ z^{n}\right\} _{n\geq0}\text{, \ \ \ }H_{-}=%
\text{the span of }\left\{ z^{n}\right\} _{n\leq-1}
\end{equation*}
in $H$. Then, $H=H_{+}\oplus H_{-}$ (orthogonal sum) holds and $\mathfrak{p}%
_{\pm}$ are the orthogonal projections to $H_{\pm}$ respectively. A closed
subspace $W$ of $H$ is $z^{2}$-\emph{invariant} if it satisfies the
condition: 
\begin{equation}
z^{2}W\subset W\text{.}  \label{6}
\end{equation}
Any $z^{2}$-invariant subspace can be identified with a shift invariant
subspace in the product space as follows. Product spaces of $H$, $H_{+}$, $%
H_{-}$ are denoted by the bold $\boldsymbol{H}$, $\boldsymbol{H}_{+}$, $%
\boldsymbol{H}_{-}$, namely%
\begin{equation*}
\boldsymbol{H}=H\times H,\text{ \ \ \ }\boldsymbol{H}_{+}=H_{+}\times H_{+},%
\text{ \ \ \ }\boldsymbol{H}_{-}=H_{-}\times H_{-}
\end{equation*}
respectively. An identity $\boldsymbol{H}=\boldsymbol{H}_{+}\oplus 
\boldsymbol{H}_{-}$ holds and the associated orthogonal projections are
denoted by $\mathfrak{p}_{\pm}$ again. For $\boldsymbol{z}=$ $^{t}\left(
z_{1},z_{2}\right) $, $\boldsymbol{w}=$ $^{t}\left( w_{1},w_{2}\right) \in%
\mathbb{C}^{2}$ denote%
\begin{equation*}
\boldsymbol{z}\cdot\boldsymbol{w}=z_{1}w_{1}+z_{2}w_{2}\text{, \ \ \ }%
\left\Vert \boldsymbol{z}\right\Vert =\sqrt{\boldsymbol{z}\cdot\overline {%
\boldsymbol{z}}}\text{, \ \ \ }\overline{\boldsymbol{z}}=\text{ }^{t}\left( 
\overline{z_{1}},\overline{z_{2}}\right) \text{.}
\end{equation*}
For a function $f\left( z\right) =\sum_{n}f_{n}z^{n}$ on $\partial \mathbb{D}%
_{r}$ set%
\begin{equation*}
f_{e}(z)=\sum_{n}f_{2n}z^{n}\text{, \ \ \ \ }f_{o}(z)=\sum_{n}f_{2n+1}z^{n}%
\text{.}
\end{equation*}
Then, we have an isomorphism%
\begin{equation}
H\ni f\rightarrow\phi\left( f\right) =\left( 
\begin{array}{c}
f_{e} \\ 
f_{o}%
\end{array}
\right) \in\boldsymbol{H}\text{.}  \label{7}
\end{equation}
Then, $\boldsymbol{W}\equiv\phi\left( W\right) $ for a $z^{2}$-invariant
subspace $W\subset H$ satisfies%
\begin{equation}
z\boldsymbol{W}\subset\boldsymbol{W}\text{.}  \label{8}
\end{equation}
A closed subspace $\boldsymbol{W}$ of $\boldsymbol{H}$ is called \emph{shift
invariant} if it satisfies (\ref{8}). This $\phi$ clearly defines an
isomorphism between $z^{2}$-invariant subspaces $W$ and shift invariant
subspaces $\boldsymbol{W}$. We identify $W$ with $\boldsymbol{W}$ from now
on.

An example of shift invariant subspace in $\boldsymbol{H}$ is given by a $%
2\times2$ non-singular matrix function $A(z)$ on $\partial\mathbb{D}_{r}$.
Assume every entry of $A(z)$, $A(z)^{-1}$ belongs to $L^{\infty}\left(
\partial\mathbb{D}_{r}\right) $, and define%
\begin{equation}
\begin{array}{l}
\boldsymbol{W}=A(z)\boldsymbol{H}_{+}=\left\{ \boldsymbol{f}\left( z\right)
=A(z)\boldsymbol{u}\left( z\right) \text{; \ \ }\boldsymbol{u}\in\boldsymbol{%
H}_{+}\right\}%
\end{array}
\text{.}  \label{9}
\end{equation}
Then, this $\boldsymbol{W}$ is a shift invariant closed subspace of $%
\boldsymbol{H}$, since so is $\boldsymbol{H}_{+}$. Denote by $Gr^{\left(
2\right) }$ the set of all $z^{2}$-invariant closed subspaces $W$ of $H$
satisfying%
\begin{equation}
\mathfrak{p}_{+}\text{: \ }W\rightarrow H_{+}\text{ \ \ is bijective.}
\label{10}
\end{equation}
$W\in Gr^{\left( 2\right) }$ if and only if $\boldsymbol{W}=\phi\left(
W\right) $ is a shift invariant closed subspace of $\boldsymbol{H}$
satisfying%
\begin{equation}
\mathfrak{p}_{+}\text{: \ }\boldsymbol{W}\rightarrow\boldsymbol{H}_{+}\text{
\ \ is bijective.}  \label{11}
\end{equation}
The property (\ref{11}) is equivalent to the invertibility of the associated
Toeplitz operator. A Toeplitz operator $T(a)$ with $a\in L^{\infty}\left(
\partial\mathbb{D}_{r}\right) $ is a bounded operator on $H_{+}$ defined by%
\begin{equation*}
T(a)f=\mathfrak{p}_{+}\left( af\right) \text{ \ \ \ \ for \ }f\in H_{+}\text{%
.}
\end{equation*}
For a bounded matrix function $A(z)$ the associated (matrix) Toeplitz
operator $\boldsymbol{T}(A)$ is defined by%
\begin{equation*}
\boldsymbol{T}(A)f=\mathfrak{p}_{+}\left( A\boldsymbol{f}\right) \text{ \ \
\ \ for \ }\boldsymbol{f}\in\boldsymbol{H}_{+}\text{.}
\end{equation*}
In (\ref{9}) $\boldsymbol{W}$ is a shift invariant closed subspace of $%
\boldsymbol{H}$, and $\boldsymbol{W}$ satisfies (\ref{11}) if and only if $%
\boldsymbol{T}(A)$ is invertible. Generally the invertibility of $%
\boldsymbol{T}(A)$ is not easy to verify. However, there is a case where one
can reduce the problem to the invertibility of a scaler Toeplitz operator as
follows. For a bounded function $m$ on $\partial\mathbb{D}_{r}$ define%
\begin{equation}
A(z)=\left( 
\begin{array}{cc}
1 & m_{e}(z) \\ 
0 & m_{o}(z)%
\end{array}
\right) \text{, \ then }\boldsymbol{T}\left( A\right) =\left( 
\begin{array}{cc}
I & T\left( m_{e}\right) \\ 
0 & T\left( m_{o}\right)%
\end{array}
\right) \text{.}  \label{12}
\end{equation}
Since%
\begin{align*}
\left( 
\begin{array}{cc}
I & T\left( m_{e}\right) \\ 
0 & T\left( m_{o}\right)%
\end{array}
\right) \left( 
\begin{array}{cc}
I & -T\left( m_{e}\right) T\left( m_{o}\right) ^{-1} \\ 
0 & T\left( m_{o}\right) ^{-1}%
\end{array}
\right) & =\left( 
\begin{array}{cc}
I & 0 \\ 
0 & I%
\end{array}
\right) \\
\left( 
\begin{array}{cc}
I & -T\left( m_{e}\right) T\left( m_{o}\right) ^{-1} \\ 
0 & T\left( m_{o}\right) ^{-1}%
\end{array}
\right) \left( 
\begin{array}{cc}
I & T\left( m_{e}\right) \\ 
0 & T\left( m_{o}\right)%
\end{array}
\right) & =\left( 
\begin{array}{cc}
I & 0 \\ 
0 & I%
\end{array}
\right) \text{,}
\end{align*}
$\boldsymbol{T}\left( A\right) $ is invertible if and only if so is $T\left(
m_{o}\right) $. A sufficient condition for the invertibility of $T\left(
m_{o}\right) $ is given in the Lemma:

\begin{lemma}
\label{l12}Let $a$ be continuous on $\mathbb{C}\backslash\mathbb{D}_{r}$
satisfying\newline
(i) \ $a\left( z\right) \neq0$ for any $z\in \mathbb{C}\backslash\mathbb{D}%
_{r}$.\newline
(ii) \ $a\left( z\right) -1\in H_{-}$.\newline
Then, $T\left( a\right) $ is invertible on $H_{+}$ and $T\left( a\right)
^{-1}=T(a^{-1})$ is valid.
\end{lemma}

\begin{proof}
Under the conditions (i), (ii) $a$ maps $H_{-}$ to $H_{-}$ bijectively. For%
\textrm{\ }$u\in H_{+}$ let%
\begin{equation*}
a(z)u(z)=f_{+}(z)+f_{-}(z)\text{ \ \ with \ }f_{\pm}\in H_{\pm}\text{.}
\end{equation*}
Then, $f_{+}=\mathfrak{p}_{+}\left( au\right) =T(a)u$, and%
\begin{equation*}
u=a^{-1}f_{+}+a^{-1}f_{-}=\mathfrak{p}_{+}\left( a^{-1}f_{+}\right)
=T(a^{-1})f_{+}=T(a^{-1})T(a)u
\end{equation*}
hence $T(a^{-1})T(a)u=u$ holds. Similarly one can prove $T(a)T(a^{-1})u=u$.$%
\medskip$
\end{proof}

An example of $a$ is given by $m$ defined in (\ref{13}) by two Weyl
functions $m_{\pm}$, and assume $m$ is holomorphic on $\mathbb{C}%
\backslash\left( \left[ -r,r\right] \cup i\left[ -r,r\right] \right) $.
Then, $T\left( m_{o}\right) $ defined on $\left\{ \left\vert z\right\vert
=s^{2}\right\} $ is invertible for any $s>r$. This is because%
\begin{equation*}
m_{o}\left( z\right) =-\dfrac{m_{+}(-z)+m_{-}(-z)}{2\sqrt{z}}\text{,}
\end{equation*}
and $m_{+}+m_{-}$ has positive imaginary part on $\mathbb{C}_{+}$.\medskip

\section{Characteristic matrix}

In this section we investigate the condition (\ref{10}) (equivalently (\ref%
{11})) by introducing characteristic matrix.

Let $H=L^{2}\left( \partial\mathbb{D}_{r}\right) $ with inner product%
\begin{equation*}
\left( f,g\right) =\dfrac{r}{2\pi}\int_{0}^{2\pi}f\left( re^{i\theta
}\right) \overline{g\left( re^{i\theta}\right) }d\theta\text{.}
\end{equation*}
Set%
\begin{equation*}
\left( J\boldsymbol{f}\right) \left( z\right) =\overline{z}\boldsymbol{f}%
\left( \overline{z}\right) \text{ \ \ for \ }\boldsymbol{f}\in \boldsymbol{H}%
\left( =H\times H\right) \text{.}
\end{equation*}
Then, $J$ maps $\boldsymbol{H}_{+}$ onto\ $\boldsymbol{H}_{-}$ and satisfies%
\begin{equation*}
J^{\ast}=J,\text{ \ }J^{2}=r^{2}I\text{ \ and \ }zJ=J\overline{z}\text{.}
\end{equation*}
Define a dual object of a closed subspace $\boldsymbol{W}$ of $\boldsymbol{H}
$ by%
\begin{equation*}
\widetilde{\boldsymbol{W}}=J\boldsymbol{W}^{\perp}\text{.}
\end{equation*}
For a closed subspace $\boldsymbol{W}\in Gr^{\left( 2\right) }$ (the correct
notation is $W\in Gr^{\left( 2\right) }$, however we abuse the notation) the
operator associating $\boldsymbol{f}_{+}\in\boldsymbol{H}_{+}$ to a unique $%
\boldsymbol{f}_{-}\in\boldsymbol{H}_{-}$ such that $\boldsymbol{f}_{+}+%
\boldsymbol{f}_{-}\in\boldsymbol{W}$ is denoted by $A_{\boldsymbol{W}}.$

\begin{lemma}
\label{l1}(i) If $\boldsymbol{W}$ satisfies $z\boldsymbol{W}\subset 
\boldsymbol{W}$, then so does $\widetilde{\boldsymbol{W}}$.\newline
(ii) If $\boldsymbol{W}$ satisfies the condition of (\ref{11}), so does $%
\widetilde {\boldsymbol{W}}$, and%
\begin{equation*}
A_{\widetilde{\boldsymbol{W}}}=-r^{-2}JA_{\boldsymbol{W}}^{\ast}J\text{.}
\end{equation*}
\end{lemma}

\begin{proof}
For $\boldsymbol{f}=J\boldsymbol{u}\in\widetilde{\boldsymbol{W}}$ with $%
\boldsymbol{u}\in\boldsymbol{W}^{\perp}$%
\begin{equation*}
z\boldsymbol{f}=zJ\boldsymbol{u}=J\left( \overline{z}\boldsymbol{u}\right)
\end{equation*}
and for $\boldsymbol{v}\in\boldsymbol{W}$%
\begin{equation*}
\left( \overline{z}\boldsymbol{u},\text{ }\boldsymbol{v}\right) =\left( 
\boldsymbol{u},\text{ }z\boldsymbol{v}\right) =0\rightarrow\overline {z}%
\boldsymbol{u}\in\boldsymbol{W}^{\perp}
\end{equation*}
hence $J\left( \overline{z}\boldsymbol{u}\right) \in J\boldsymbol{W}^{\perp} 
$, and $z\boldsymbol{f}\in\widetilde{\boldsymbol{W}}$, which proves (i).%
\newline
It follows from $\boldsymbol{W}=\left\{ \boldsymbol{f}+A_{\boldsymbol{W}}%
\boldsymbol{f}\text{; \ }\boldsymbol{f}\in\boldsymbol{H}_{+}\right\} $ that%
\begin{equation*}
\boldsymbol{W}^{\perp}=\left\{ \boldsymbol{u}-A_{\boldsymbol{W}}^{\ast }%
\boldsymbol{u}\text{; \ }\boldsymbol{u}\in\boldsymbol{H}_{-}\right\} \text{.}
\end{equation*}
Hence%
\begin{equation*}
\widetilde{\boldsymbol{W}}\ni J\boldsymbol{u}-JA_{\boldsymbol{W}}^{\ast }%
\boldsymbol{u}=J\boldsymbol{u}-r^{-2}JA_{\boldsymbol{W}}^{\ast}JJ\boldsymbol{%
u}\text{ \ \ for }\boldsymbol{u}\in\boldsymbol{H}_{-}\text{.}
\end{equation*}
Since $J:$ $\boldsymbol{H}_{-}\rightarrow\boldsymbol{H}_{+}$ (bijective),
the above identity shows that $\widetilde{\boldsymbol{W}}$ satisfies (\ref%
{11}) and simultaneously $A_{\widetilde{\boldsymbol{W}}}=-r^{-2}JA_{%
\boldsymbol{W}}^{\ast}J$ holds.\medskip
\end{proof}

Let%
\begin{equation*}
\boldsymbol{e}_{1}=\left( 
\begin{array}{c}
1 \\ 
0%
\end{array}
\right) \text{, \ \ \ }\boldsymbol{e}_{2}=\left( 
\begin{array}{c}
0 \\ 
1%
\end{array}
\right) \in\boldsymbol{H}_{+}\text{.}
\end{equation*}
For $\boldsymbol{W}\in Gr^{\left( 2\right) }$ set%
\begin{equation*}
\boldsymbol{\varphi}_{\boldsymbol{W}}=A_{\boldsymbol{W}}\boldsymbol{e}_{1}%
\text{, \ \ }\boldsymbol{\psi}_{\boldsymbol{W}}=A_{\boldsymbol{W}}%
\boldsymbol{e}_{2}\in\boldsymbol{H}_{-}\text{,}
\end{equation*}
and define a $2\times2$ matrix by%
\begin{equation*}
\mathit{\Pi}_{\boldsymbol{W}}\left( z\right) =\left[ \boldsymbol{e}_{1}+%
\boldsymbol{\varphi}_{\boldsymbol{W}}\left( z\right) \text{, }\boldsymbol{e}%
_{2}+\boldsymbol{\psi}_{\boldsymbol{W}}\left( z\right) \right] =I+\left[ 
\boldsymbol{\varphi}_{\boldsymbol{W}}\left( z\right) \text{, }\boldsymbol{%
\psi}_{\boldsymbol{W}}\left( z\right) \right] \text{.}
\end{equation*}

\begin{proposition}
\label{p2}For $\boldsymbol{W}\in Gr^{\left( 2\right) }$ the matrices $%
\mathit{\Pi}_{\boldsymbol{W}}\left( z\right) $, $\mathit{\Pi}_{\widetilde{%
\boldsymbol{W}}}\left( z\right) $ satisfy the following properties:\newline
(i) For $\boldsymbol{f}\in\boldsymbol{H}_{+}$%
\begin{equation}
A_{\boldsymbol{W}}z\boldsymbol{f}=zA_{\boldsymbol{W}}\boldsymbol{f+}%
r^{-1}\left( J\boldsymbol{f}\text{, }\boldsymbol{\varphi}_{\widetilde {%
\boldsymbol{W}}}\right) \left( \boldsymbol{e}_{1}+\boldsymbol{\varphi }_{%
\boldsymbol{W}}\right) +r^{-1}\left( J\boldsymbol{f}\text{, }\boldsymbol{\psi%
}_{\widetilde{\boldsymbol{W}}}\right) \left( \boldsymbol{e}_{2}+\boldsymbol{%
\psi}_{\boldsymbol{W}}\right) \text{.}  \label{14}
\end{equation}
(ii) $\mathit{\Pi}_{\boldsymbol{W}}\left( z\right) $ is invertible for any $%
z $ such that $\left\vert z\right\vert >r$ and for a.e. $z$ in $\partial 
\mathbb{D}_{r}$. Moreover, every entry of $\mathit{\Pi}_{\boldsymbol{W}%
}\left( z\right) $, $\mathit{\Pi}_{\boldsymbol{W}}\left( z\right) ^{-1}$
belongs to $\overline{H_{+}}$. Additionally $\widetilde{\boldsymbol{W}}\in
Gr^{\left( 2\right) }$ is valid and it holds that%
\begin{equation}
\mathit{\Pi}_{\boldsymbol{W}}\left( z\right) \text{ }^{t}\overline {\mathit{%
\Pi}_{\widetilde{\boldsymbol{W}}}\left( \overline{z}\right) }=I\text{.}
\label{15}
\end{equation}
\end{proposition}

\begin{proof}
For $\boldsymbol{f}\in\boldsymbol{H}_{+}$%
\begin{equation*}
\boldsymbol{W}\ni z\boldsymbol{f}+zA_{\boldsymbol{W}}\boldsymbol{f}=\left( z%
\boldsymbol{f}+\mathfrak{p}_{+}zA_{\boldsymbol{W}}\boldsymbol{f}\right) +%
\mathfrak{p}_{-}zA_{\boldsymbol{W}}\boldsymbol{f}\text{,}
\end{equation*}
hence%
\begin{equation*}
A_{\boldsymbol{W}}\left( z\boldsymbol{f}+\mathfrak{p}_{+}zA_{\boldsymbol{W}}%
\boldsymbol{f}\right) =\mathfrak{p}_{-}zA_{\boldsymbol{W}}\boldsymbol{f}=zA_{%
\boldsymbol{W}}\boldsymbol{f}-\mathfrak{p}_{+}zA_{\boldsymbol{W}}\boldsymbol{%
f}\text{,}
\end{equation*}
thus%
\begin{equation}
A_{\boldsymbol{W}}z\boldsymbol{f}=zA_{\boldsymbol{W}}\boldsymbol{f}-%
\mathfrak{p}_{+}zA_{\boldsymbol{W}}\boldsymbol{f}-A_{\boldsymbol{W}}%
\mathfrak{p}_{+}zA_{\boldsymbol{W}}\boldsymbol{f}\text{.}  \label{16}
\end{equation}
Since, for $\boldsymbol{u}=A_{\boldsymbol{W}}\boldsymbol{f}\in\boldsymbol{H}%
_{-}$%
\begin{equation*}
\boldsymbol{u}=\sum_{n\geq1}\boldsymbol{u}_{n}z^{-n}\rightarrow\mathfrak{p}%
_{+}\left( z\boldsymbol{u}\right) =\boldsymbol{u}_{1}=r^{-1}\left( z%
\boldsymbol{u},\boldsymbol{e}_{1}\right) \boldsymbol{e}_{1}+r^{-1}\left( z%
\boldsymbol{u},\boldsymbol{e}_{2}\right) \boldsymbol{e}_{2}
\end{equation*}%
\begin{align*}
\mathfrak{p}_{+}\left( z\boldsymbol{u}\right) & =r^{-1}\left( \boldsymbol{u}%
\text{, }\overline{z}\boldsymbol{e}_{1}\right) \boldsymbol{e}%
_{1}+r^{-1}\left( \boldsymbol{u}\text{, }\overline{z}\boldsymbol{e}%
_{2}\right) \boldsymbol{e}_{2} \\
& =r^{-1}\left( A_{\boldsymbol{W}}\boldsymbol{f}\text{, }\overline {z}%
\boldsymbol{e}_{1}\right) \boldsymbol{e}_{1}+r^{-1}\left( A_{\boldsymbol{W}}%
\boldsymbol{f}\text{, }\overline{z}\boldsymbol{e}_{2}\right) \boldsymbol{e}%
_{2} \\
& =r^{-3}\left( J\boldsymbol{f}\text{, }JA_{\boldsymbol{W}}^{\ast }J%
\boldsymbol{e}_{1}\right) \boldsymbol{e}_{1}+r^{-3}\left( J\boldsymbol{f}%
\text{, }JA_{\boldsymbol{W}}^{\ast}J\boldsymbol{e}_{2}\right) \boldsymbol{e}%
_{2} \\
& =-r^{-1}\left( J\boldsymbol{f}\text{, }A_{\widetilde{\boldsymbol{W}}}%
\boldsymbol{e}_{1}\right) \boldsymbol{e}_{1}-r^{-1}\left( J\boldsymbol{f}%
\text{, }A_{\widetilde{\boldsymbol{W}}}\boldsymbol{e}_{2}\right) \boldsymbol{%
e}_{2} \\
& =-r^{-1}\left( J\boldsymbol{f}\text{, }\boldsymbol{\varphi}_{\widetilde {%
\boldsymbol{W}}}\right) \boldsymbol{e}_{1}-r^{-1}\left( J\boldsymbol{f}\text{%
, }\boldsymbol{\psi}_{\widetilde{\boldsymbol{W}}}\right) \boldsymbol{e}_{2}%
\text{,}
\end{align*}
thus (\ref{14}) follows from (\ref{16}).\newline
Since $\widetilde {\boldsymbol{W}}\in Gr^{\left( 2\right) }$ follows from
Lemma\ref{l1}, we have only to show (\ref{15}). Applying (\ref{16}) to $%
\boldsymbol{f}=z^{n}\boldsymbol{e}_{1}$ $\left( n\geq0\right) $ yields%
\begin{align*}
& A_{\boldsymbol{W}}z^{n+1}\boldsymbol{e}_{1} \\
& =zA_{\boldsymbol{W}}z^{n}\boldsymbol{e}_{1}+r^{-1}\left( Jz^{n}\boldsymbol{%
e}_{1}\text{, }\boldsymbol{\varphi}_{\widetilde{\boldsymbol{W}}}\right)
\left( \boldsymbol{e}_{1}+\boldsymbol{\varphi}_{\boldsymbol{W}}\right)
+r^{-1}\left( Jz^{n}\boldsymbol{e}_{1}\text{, }\boldsymbol{\psi }_{%
\widetilde{\boldsymbol{W}}}\right) \left( \boldsymbol{e}_{2}+\boldsymbol{\psi%
}_{\boldsymbol{W}}\right) \\
& =zA_{\boldsymbol{W}}z^{n}\boldsymbol{e}_{1}+r^{-1}\left( \overline {z}%
^{n+1}\boldsymbol{e}_{1}\text{, }\boldsymbol{\varphi}_{\widetilde {%
\boldsymbol{W}}}\right) \left( \boldsymbol{e}_{1}+\boldsymbol{\varphi }_{%
\boldsymbol{W}}\right) +r^{-1}\left( \overline{z}^{n+1}\boldsymbol{e}_{1}%
\text{, }\boldsymbol{\psi}_{\widetilde{\boldsymbol{W}}}\right) \left( 
\boldsymbol{e}_{2}+\boldsymbol{\psi}_{\boldsymbol{W}}\right) \text{,}
\end{align*}
similarly%
\begin{align*}
& A_{\boldsymbol{W}}z^{n+1}\boldsymbol{e}_{2} \\
& =zA_{\boldsymbol{W}}z^{n}\boldsymbol{e}_{2}+r^{-1}\left( \overline {z}%
^{n+1}\boldsymbol{e}_{2}\text{, }\boldsymbol{\varphi}_{\widetilde {%
\boldsymbol{W}}}\right) \left( \boldsymbol{e}_{1}+\boldsymbol{\varphi }_{%
\boldsymbol{W}}\right) +r^{-1}\left( \overline{z}^{n+1}\boldsymbol{e}_{2}%
\text{, }\boldsymbol{\psi}_{\widetilde{\boldsymbol{W}}}\right) \left( 
\boldsymbol{e}_{2}+\boldsymbol{\psi}_{\boldsymbol{W}}\right) \text{,}
\end{align*}
hence, for $j=1,2$%
\begin{equation}
A_{\boldsymbol{W}}z^{n}\boldsymbol{e}_{j}=z^{n}\left( A_{\boldsymbol{W}}%
\boldsymbol{e}_{j}+\left( \boldsymbol{e}_{1}+\boldsymbol{\varphi }_{%
\boldsymbol{W}}\right) \overline{\boldsymbol{\varphi}}_{\widetilde {%
\boldsymbol{W}},j,n}+\left( \boldsymbol{e}_{2}+\boldsymbol{\psi }_{%
\boldsymbol{W}}\right) \overline{\boldsymbol{\psi}}_{\widetilde {\boldsymbol{%
W}},j,n}\right) \text{,}  \label{17}
\end{equation}
holds, where%
\begin{equation*}
\boldsymbol{\varphi}_{\widetilde{\boldsymbol{W}},j,n}\left( z\right)
=r^{-1}\sum_{k=1}^{n}z^{-k}\left( \boldsymbol{\varphi}_{\widetilde {%
\boldsymbol{W}}}\text{, }\overline{z}^{k}\boldsymbol{e}_{j}\right) \text{, \ 
}\boldsymbol{\psi}_{\widetilde{\boldsymbol{W}},j,n}\left( z\right)
=r^{-1}\sum_{k=1}^{n}z^{-k}\left( \boldsymbol{\psi}_{\widetilde {\boldsymbol{%
W}}}\text{, }\overline{z}^{k}\boldsymbol{e}_{j}\right) \text{,}
\end{equation*}
and generally $\overline{f}\left( z\right) =\overline{f\left( \overline {z}%
\right) }$. In a matrix form the identity (\ref{17}) turns out to be%
\begin{equation*}
z^{-n}\left[ A_{\boldsymbol{W}}z^{n}\boldsymbol{e}_{1}\text{, }A_{%
\boldsymbol{W}}z^{n}\boldsymbol{e}_{2}\right] =\mathit{\Pi}_{\boldsymbol{W}%
}\left( z\right) \text{ }^{t}\overline{\mathit{\Pi }_{\widetilde{\boldsymbol{%
W}},n}\left( \overline{z}\right) }-I
\end{equation*}
\ for any $n\geq0$, where%
\begin{equation*}
\mathit{\Pi}_{\widetilde{\boldsymbol{W}},n}\left( z\right) =\left[ 
\begin{array}{cc}
1+\boldsymbol{\varphi}_{\widetilde{\boldsymbol{W}},1,n}\left( z\right) & 
\boldsymbol{\psi}_{\widetilde{\boldsymbol{W}},1,n}\left( z\right) \\ 
\boldsymbol{\varphi}_{\widetilde{\boldsymbol{W}},2,n}\left( z\right) & 1+%
\boldsymbol{\psi}_{\widetilde{\boldsymbol{W}},2,n}\left( z\right)%
\end{array}
\right] \text{.}
\end{equation*}
Noting%
\begin{equation*}
\left( 
\begin{array}{c}
\boldsymbol{\varphi}_{\widetilde{\boldsymbol{W}},1,n}\left( z\right) \\ 
\boldsymbol{\varphi}_{\widetilde{\boldsymbol{W}},2,n}\left( z\right)%
\end{array}
\right) \rightarrow\boldsymbol{\varphi}_{\widetilde{\boldsymbol{W}}}\left(
z\right) \text{, \ }\left( 
\begin{array}{c}
\boldsymbol{\psi}_{\widetilde{\boldsymbol{W}},1,n}\left( z\right) \\ 
\boldsymbol{\psi}_{\widetilde{\boldsymbol{W}},2,n}\left( z\right)%
\end{array}
\right) \rightarrow\boldsymbol{\psi}_{\widetilde{\boldsymbol{W}}}\left(
z\right) \text{ in }\boldsymbol{H}_{-}
\end{equation*}
as $n\rightarrow\infty$ and%
\begin{equation*}
\left\Vert A_{\boldsymbol{W}}z^{n}\boldsymbol{e}_{1}\right\Vert
\leq\left\Vert A_{\boldsymbol{W}}\right\Vert \left\Vert z^{n}\boldsymbol{e}%
_{1}\right\Vert =r^{n+1/2}\left\Vert A_{\boldsymbol{W}}\right\Vert \text{,}
\end{equation*}
we see%
\begin{equation*}
\mathit{\Pi}_{\boldsymbol{W}}\left( z\right) \text{ }^{t}\overline {\mathit{%
\Pi}_{\widetilde{\boldsymbol{W}}}\left( \overline{z}\right) }-I=0\text{ \
for }z\text{ such that }\left\vert z\right\vert >r\text{,}
\end{equation*}
by letting $n\rightarrow\infty$ in (\ref{17}), which completes the proof of (%
\ref{15}) by letting $\left\vert z\right\vert \rightarrow r$.\medskip
\end{proof}

The identities (\ref{14}), (\ref{15}) show that the operator $A_{\boldsymbol{%
W}}$ is uniquely determined by $\left\{ \boldsymbol{\varphi }_{\boldsymbol{W}%
}\text{, }\boldsymbol{\psi}_{\boldsymbol{W}}\right\} $, since $\boldsymbol{H}%
_{+}$ is generated by $\left\{ z^{m}\boldsymbol{e}_{1}\text{, }z^{n}%
\boldsymbol{e}_{2}\right\} _{m,n\geq0}$. This implies that $\boldsymbol{W}%
\in Gr^{\left( 2\right) }$ is uniquely determined by $\mathit{\Pi}_{%
\boldsymbol{W}}$, and we call $\mathit{\Pi}_{\boldsymbol{W}}$ as the \emph{%
characteristic matrix} of $\boldsymbol{W}$ (or $W$).

\section{Group action on $Gr^{\left( 2\right) }$ and $\protect\tau$-function}

In this section we consider an commutative action on $Gr^{\left( 2\right) }$%
. Set%
\begin{equation}
\mathit{\Gamma}=\left\{ g=e^{h}\text{; }h\text{ holomorphic on }\mathbb{D}%
_{s}\text{ for some }s>r\text{ and }h(0)=0\right\} \text{.}  \label{46}
\end{equation}
Then $\mathit{\Gamma}$ is commutative and we can consider a closed subspace $%
gW$ for $W\in Gr^{\left( 2\right) }$. To investigate this action we define
an operator $R_{W}$ on $H_{+}$ for $W\in Gr^{\left( 2\right) }$ by%
\begin{equation*}
R_{W}\left( g\right) =g^{-1}\mathfrak{p}_{+}gA_{W}\text{.}
\end{equation*}
Then, the $\tau$-function is $\tau_{W}(g)=\det\left( I+R_{W}\left( g\right)
\right) $. To define this determinant the traceability of $R_{W}\left(
g\right) $ is required.

\begin{lemma}
\label{l2}Suppose $g_{j}\in\mathit{\Gamma}$ for $j=1,2$. Then%
\begin{align*}
& \left\Vert g_{1}^{-1}\mathfrak{p}_{+}g_{1}-g_{2}^{-1}\mathfrak{p}%
_{+}g_{2}\right\Vert _{trace} \\
& \leq3r^{-1/2}\left( 
\begin{array}{c}
\left\Vert g_{1}^{-1}-g_{2}^{-1}\right\Vert \left( \left\Vert
g_{1}-1\right\Vert +r^{2}\left\Vert g_{1}^{\prime\prime}\right\Vert \right)
\\ 
+\left\Vert g_{2}^{-1}\right\Vert \left( \left\Vert g_{1}-g_{2}\right\Vert
+r^{2}\left\Vert g_{1}^{\prime\prime}-g_{2}^{\prime\prime}\right\Vert \right)%
\end{array}
\right) \text{.}
\end{align*}
\end{lemma}

\begin{proof}
For $f\in H_{-}$ and $g_{j}\in H^{2}\left( \partial\mathbb{D}_{r}\right) $
let%
\begin{equation*}
f\left( z\right) =\sum_{n\geq1}f_{n}z^{-n}\text{, }g_{j}\left( z\right)
=\sum_{n\geq0}g_{jn}z^{n}\text{.}
\end{equation*}
Then%
\begin{align*}
& \partial_{\theta}\left( \mathfrak{p}_{+}g_{1}f-\mathfrak{p}%
_{+}g_{2}f\right) \left( re^{i\theta}\right) \\
& =i\sum_{m\geq1,n-m\geq0}\left( n-m\right) r^{n-m}\left(
g_{1,n}-g_{2,n}\right) f_{m}e^{i\left( n-m\right) \theta} \\
& =i\sum_{m\geq1}f_{m}\sum_{k\geq0}kr^{k}e^{ik\theta}\left(
g_{1,k+m}-g_{2,k+m}\right) \text{,}
\end{align*}
hence%
\begin{align*}
& \left\Vert \left( 1+\partial_{\theta}\right) \left( \mathfrak{p}_{+}g_{1}-%
\mathfrak{p}_{+}g_{2}\right) \right\Vert _{HS}^{2} \\
& =\sum_{m\geq1}\dfrac{r}{2\pi}\int_{0}^{2\pi}\left\vert r^{m-1/2}\sum
_{k\geq0}\left( 1+ik\right) r^{k}e^{ik\theta}\left(
g_{1,k+m}-g_{2,k+m}\right) \right\vert ^{2}d\theta \\
& =\sum_{m\geq1}\sum_{k\geq m}r^{2k}\left( 1+\left( k-m\right) ^{2}\right)
\left\vert g_{1,k}-g_{2,k}\right\vert
^{2}\leq\sum_{k\geq1}r^{2k}k^{3}\left\vert g_{1,k}-g_{2,k}\right\vert ^{2}%
\text{,}
\end{align*}
since $\sum_{m=1}^{k}\left( 1+\left( k-m\right) ^{2}\right) \leq k^{3}$
holds if $k\geq1$. Note%
\begin{align*}
& \sum_{k\geq1}r^{2k}k^{3}\left\vert g_{1,k}-g_{2,k}\right\vert ^{2} \\
& \leq\left\vert rg_{1,1}-rg_{2,1}\right\vert ^{2}+2\sum_{k\geq2}k^{2}\left(
k-1\right) ^{2}\left\vert r^{k}g_{1,k}-r^{k}g_{2,k}\right\vert ^{2} \\
& \leq2r^{-1}\left( \left\Vert g_{1}-g_{2}\right\Vert ^{2}+r^{4}\left\Vert
g_{1}^{\prime\prime}-g_{2}^{\prime\prime}\right\Vert ^{2}\right) \text{.}
\end{align*}
Since%
\begin{equation*}
\left\Vert \left( 1+\partial_{\theta}\right) ^{-1}\right\Vert
_{HS}^{2}=\sum_{k\geq0}\left\vert 1+ik\right\vert ^{-2}\leq1+\dfrac{\pi^{2}}{%
6}<4<\infty\text{,}
\end{equation*}
thus%
\begin{align*}
& \left\Vert \mathfrak{p}_{+}g_{1}-\mathfrak{p}_{+}g_{2}\right\Vert _{trace}
\\
& \leq\left\Vert \left( 1+\partial_{\theta}\right) ^{-1}\right\Vert
_{HS}\left\Vert \left( 1+\partial_{\theta}\right) \left( \mathfrak{p}%
_{+}g_{1}-\mathfrak{p}_{+}g_{2}\right) \right\Vert _{HS} \\
& \leq2\sqrt{2}r^{-1/2}\left( \left\Vert g_{1}-g_{2}\right\Vert
+r^{2}\left\Vert g_{1}^{\prime\prime}-g_{2}^{\prime\prime}\right\Vert
\right) \text{.}
\end{align*}
Consequently%
\begin{align*}
& \left\Vert g_{1}^{-1}\mathfrak{p}_{+}g_{1}-g_{2}^{-1}\mathfrak{p}%
_{+}g_{2}\right\Vert _{trace} \\
& \leq\left\Vert g_{1}^{-1}-g_{2}^{-1}\right\Vert \left\Vert \mathfrak{p}%
_{+}g_{1}\right\Vert _{trace}+\left\Vert g_{2}^{-1}\right\Vert \left\Vert 
\mathfrak{p}_{+}g_{1}-\mathfrak{p}_{+}g_{2}\right\Vert _{trace} \\
& \leq3r^{-1/2}\left( 
\begin{array}{c}
\left\Vert g_{1}^{-1}-g_{2}^{-1}\right\Vert \left( \left\Vert
g_{1}-1\right\Vert +r^{2}\left\Vert g_{1}^{\prime\prime}\right\Vert \right)
\\ 
+\left\Vert g_{2}^{-1}\right\Vert \left( \left\Vert g_{1}-g_{2}\right\Vert
+r^{2}\left\Vert g_{1}^{\prime\prime}-g_{2}^{\prime\prime}\right\Vert \right)%
\end{array}
\right) \text{,}
\end{align*}
which completes the proof.\medskip
\end{proof}

This Lemma shows that $R_{W}\left( g\right) $ is of trace class and $%
\tau_{W}(g)$ can be defined if $g\in\mathit{\Gamma}$. This $\tau_{W}(g)$ is
called as $\tau$-\emph{function} and plays a crucial role in Sato's theory.

\begin{lemma}
\label{l3}$gW\in Gr^{(2)}$ holds if and only if $\ker\left( I+R_{W}\left(
g\right) \right) =\left\{ 0\right\} $ is valid. In this case, the $A$%
-operator corresponding to $gW$ is given by%
\begin{equation*}
A_{gW}=\mathfrak{p}_{-}g^{-1}A_{W}\left( I+R_{W}\left( g\right) \right)
^{-1}g\text{.}
\end{equation*}
\end{lemma}

\begin{proof}
Since $R_{W}\left( g\right) $ is a compact operator, $\ker\left(
I+R_{W}\left( g\right) \right) =\left\{ 0\right\} $ implies the existence of 
$\left( I+R_{W}\left( g\right) \right) ^{-1}$ as a bounded operator on $%
H_{+}.$ Set%
\begin{equation*}
B=\mathfrak{p}_{-}g^{-1}A_{W}\left( I+R_{W}\left( g\right) \right) ^{-1}g%
\text{.}
\end{equation*}
For $f\in H_{+},$ an identity%
\begin{align}
& g^{-1}\left( f+Bf\right)  \notag \\
& =g^{-1}f+g^{-1}\mathfrak{p}_{-}gA_{W}\left( I+R_{W}\left( g\right) \right)
^{-1}g^{-1}f  \notag \\
& =g^{-1}f+A_{W}\left( I+R_{W}\left( g\right) \right)
^{-1}g^{-1}f-R_{W}\left( g\right) \left( I+R_{W}\left( g\right) \right)
^{-1}g^{-1}f  \notag \\
& =\left( I+R_{W}\left( g\right) \right) ^{-1}g^{-1}f+A_{W}\left(
I+R_{W}\left( g\right) \right) ^{-1}g^{-1}f  \label{18}
\end{align}
is valid, hence $g^{-1}\left( f+Bf\right) \in W$ and $gW\supset\left\{ f+Bf;%
\text{ \ }f\in H_{+}\right\} $ holds. Conversely, $f=g\left( I+R_{W}\left(
g\right) \right) u\in H_{+}$ for $u\in H_{+}$ satisfies $g^{-1}\left(
f+Bf\right) =u+A_{W}u\in W$ due to (\ref{18}). Hence we have%
\begin{equation*}
gW=\left\{ f+Bf;\text{ \ }f\in H_{+}\right\} \text{,}
\end{equation*}
which implies $gW\in Gr^{\left( 2\right) }$ and $B=A_{gW}$.\medskip
\end{proof}

Now suppose $f\in\ker\left( I+R_{W}\left( g\right) \right) ,$ then, from an
identity $gf+\mathfrak{p}_{+}\left( gA_{W}f\right) =0$ it follows that%
\begin{equation*}
\mathfrak{p}_{-}\left( gA_{W}f\right) =gA_{W}f-\mathfrak{p}_{+}\left(
gA_{W}f\right) =g\left( f+A_{W}f\right) \in gW\text{.}
\end{equation*}
Therefor $gW\in Gr^{(2)}$ implies $\mathfrak{p}_{-}\left( gA_{W}f\right) =0$%
. Hence, $f+A_{W}f=0$ and $f=0$ holds, which completes the proof.\medskip

\begin{proposition}
\label{p3}$\tau_{W}(g)$ satisfies the following properties.\newline
(i) $\ gW\in Gr^{\left( 2\right) }$ holds for $g\in\mathit{\Gamma}$ and $%
W\in Gr^{\left( 2\right) }$ if and only if $\tau_{W}(g)\neq0$.\newline
(ii) For $g_{1}$, $g_{2}\in\mathit{\Gamma}$, $W\in Gr^{\left( 2\right) }$
suppose $g_{1}W\in Gr^{\left( 2\right) }$ (equivalently $\tau_{W}\left(
g_{1}\right) \neq0$). Then%
\begin{equation}
\tau_{W}\left( g_{1}g_{2}\right) =\tau_{W}\left( g_{1}\right) \tau
_{g_{1}W}\left( g_{2}\right) \text{ \ \ \ }(\text{cocycle property})\text{.}
\label{19}
\end{equation}
(iii) If $g=e^{h}\in\mathit{\Gamma}$, and $g_{1}(z)=e^{h_{e}(z^{2})}$, $%
g_{2}(z)=e^{zh_{o}(z^{2})}$, then%
\begin{equation*}
\tau_{W}\left( g\right) =\tau_{W}\left( g_{1}\right) \tau_{W}\left(
g_{2}\right) \text{.}
\end{equation*}
(iv) $\tau_{W}(g)$ is continuous on $\mathit{\Gamma}$ with Sobolev $H^{2}$%
-norm.
\end{proposition}

\begin{proof}
Since $\tau_{W}(g)=0$ if and only if $\ker\left( I+R_{W}\left( g\right)
\right) =\left\{ 0\right\} $, (i) is valid by Lemma \ref{l3}. To show (ii)
note%
\begin{align*}
\left( g_{1}g_{2}\right) A_{W} & =g_{2}\mathfrak{p}_{+}g_{1}A_{W}+g_{2}%
\mathfrak{p}_{-}g_{1}A_{W} \\
& =g_{2}g_{1}R_{W}\left( g_{1}\right) +g_{2}A_{g_{1}W}g_{1}\left(
I+R_{W}\left( g_{1}\right) \right) \text{,}
\end{align*}
and%
\begin{equation*}
\mathfrak{p}_{+}\left( \left( g_{1}g_{2}\right) A_{W}\right)
=g_{2}g_{1}R_{W}\left( g_{1}\right) +\mathfrak{p}_{+}\left(
g_{2}A_{g_{1}W}g_{1}\left( I+R_{W}\left( g_{1}\right) \right) \right) \text{,%
}
\end{equation*}
hence%
\begin{align*}
I+R_{W}\left( g_{1}g_{2}\right) & =I+R_{W}\left( g_{1}\right)
+g_{1}^{-1}g_{2}^{-1}\mathfrak{p}_{+}\left( g_{2}A_{g_{1}W}g_{1}\left(
I+R_{W}\left( g_{1}\right) \right) \right) \\
& =I+R_{W}\left( g_{1}\right) +g_{1}^{-1}R_{g_{1}W}(g_{2})g_{1}\left(
I+R_{W}\left( g_{1}\right) \right) \\
& =g_{1}^{-1}\left( I+R_{g_{1}W}(g_{2})\right) g_{1}\left( I+R_{W}\left(
g_{1}\right) \right) \text{.}
\end{align*}
Consequently, if $g_{1}W\in Gr^{(2)}$, then we have (ii). (iii) follows
immediately from (ii) if we notice $g_{1}W=W$. (iv) is a direct consequence
of Lemma\ref{l2}.\medskip
\end{proof}

The entries of characteristic matrices can be obtained from $\tau_{W}(g)$ by
choosing $g$ appropriately. Let%
\begin{equation*}
\varphi_{W}=A_{W}1\text{, \ }\psi_{W}=A_{W}z\in H_{-}\text{, \ (then }%
\boldsymbol{\varphi}_{\boldsymbol{W}}=\text{ }^{t}\left(
\varphi_{e},\varphi_{o}\right) \text{, \ }\boldsymbol{\psi}_{\boldsymbol{W}}=%
\text{ }^{t}\left( \psi_{e},\psi_{o}\right) \text{).}
\end{equation*}
An element $q_{\zeta}$ of $\mathit{\Gamma}$ defined by%
\begin{equation*}
q_{\zeta}\left( z\right) =\left( 1-z\zeta^{-1}\right) ^{-1}
\end{equation*}
for $\zeta\in\mathbb{C}$ such that $\left\vert \zeta\right\vert >r$ plays a
crucial role in the $\mathit{\Gamma}$-action, since any $g\in\mathit{\Gamma}$
can be expressed as a limit of $q_{\zeta_{1}}q_{\zeta_{2}}\cdots
q_{\zeta_{n}}$. For $f\in H_{-}$ we have a decomposition of $q_{\zeta}f$
into $H_{-}\oplus H_{+}$%
\begin{equation*}
q_{\zeta}f\left( z\right) =\left( 1-\frac{z}{\zeta}\right) ^{-1}f(z)=\left(
1-\frac{z}{\zeta}\right) ^{-1}\left( f(z)-f\left( \zeta\right) \right)
+\left( 1-\frac{z}{\zeta}\right) ^{-1}f\left( \zeta\right) \text{,}
\end{equation*}
which yields%
\begin{equation*}
\left( q_{\zeta}^{-1}\mathfrak{p}_{+}q_{\zeta}f\right) \left( z\right)
=\left( 1-\frac{z}{\zeta}\right) \left( 1-\frac{z}{\zeta}\right)
^{-1}f\left( \zeta\right) =f(\zeta)\text{.}
\end{equation*}
Hence, if $W\in Gr^{\left( 2\right) }$, then for $f\in H_{+}$%
\begin{equation*}
\left( q_{\zeta}^{-1}\mathfrak{p}_{+}q_{\zeta}A_{W}f\right) \left( z\right)
=\left( A_{W}f\right) \left( \zeta\right) =\left( A_{W}f\right) \left(
\zeta\right) 1
\end{equation*}
holds, which implies that $q_{\zeta}^{-1}\mathfrak{p}_{+}q_{\zeta}A_{W}$ is
a linear operator of rank $1$. Thus%
\begin{equation}
\tau_{W}\left( q_{\zeta}\right) =\det\left( I+q_{\zeta}^{-1}\mathfrak{p}%
_{+}q_{\zeta}A_{W}\right) =1+\left( A_{W}1\right) (\zeta)=1+\varphi
_{W}\left( \zeta\right)  \label{20}
\end{equation}
follows. Further calculations on $\tau$-functions can be found in the
Appendix.

The next role of the $\tau$-function is to express a potential of Schr\"{o}%
dinger operator for a given $W\in Gr^{(2)}$. Let $\left\{ e_{x}\right\}
_{x\in\mathbb{C}}$ be a one parameter group of elements of $\mathit{\Gamma} $
defined by 
\begin{equation*}
e_{x}(z)=e^{xz}\text{.}
\end{equation*}

\begin{lemma}
\label{l4}Suppose $\tau_{W}\left( e_{x}\right) \neq0$ for $x$ in a domain $D$
of $\mathbb{C}$. Then, the function%
\begin{equation}
f_{W}(x,\zeta)=e^{-x\zeta}\left( 1+\varphi_{e_{x}W}\left( \zeta\right)
\right) =e^{-x\zeta}\tau_{e_{x}W}\left( q_{\zeta}\right) \in W\text{ \ (as a
function of }\zeta\text{)}  \label{22}
\end{equation}
satisfies%
\begin{equation}
-f_{W}^{\prime\prime}(x,z)+q_{W}(x)f_{W}(x,z)=-z^{2}f_{W}(x,z)\text{,}
\label{23}
\end{equation}
namely $f_{W}$ is a Baker-Akhiezer function for the Schr\"{o}dinger operator 
$L_{q_{W}}$ in $D$. Moreover, if $\left\{ a_{n}(x)\right\} _{n\geq1}$ is
defined as coefficients of an expansion%
\begin{equation*}
f_{W}(x,\cdot)=e^{-xz}\left( 1+\dfrac{a_{1}(x)}{z}+\dfrac{a_{2}(x)}{z^{2}}%
+\cdots\right) \text{,}
\end{equation*}
then%
\begin{equation}
a_{1}(x)=\partial_{x}\log\tau_{W}\left( e_{x}\right) \text{,}  \label{24}
\end{equation}
and $q_{W}$ is given by%
\begin{equation}
q_{W}(x)=-2\partial_{x}^{2}\log\tau_{W}\left( e_{x}\right) \text{.}
\label{21}
\end{equation}
\end{lemma}

\begin{proof}
Since $e_{x}W\in Gr^{\left( 2\right) }$ for $x\in D$, there exists uniquely $%
u\in e_{x}W$ such that $\mathfrak{p}_{+}u=1$. Set $f_{W}(x,z)=e_{x}^{-1}u\in
W$. Then, the calculation in the Introduction shows that $f_{W}$ satisfies
the equation (\ref{23}). The formula (\ref{24}) is verified as follows. Since%
\begin{equation*}
e^{x\zeta}f_{W}(x,\zeta)=1+\left( A_{e_{x}W}1\right) \left( \zeta\right)
=\tau_{e_{x}W}\left( q_{\zeta}\right) =\dfrac{\tau_{W}\left( e_{x}q_{\zeta
}\right) }{\tau_{W}\left( e_{x}\right) }
\end{equation*}
is valid, we see%
\begin{align*}
a_{1}(x) & =\lim_{\zeta\rightarrow\infty}\zeta\left(
e^{x\zeta}f_{W}(x,\zeta)-1\right) =\lim_{\zeta\rightarrow\infty}\zeta\dfrac{%
\tau _{W}\left( e_{x}q_{\zeta}\right) -\tau_{W}\left( e_{x}\right) }{\tau
_{W}\left( e_{x}\right) } \\
& =\lim_{\zeta\rightarrow\infty}\zeta\dfrac{\tau_{W}\left(
e_{x+\zeta^{-1}}\right) -\tau_{W}\left( e_{x}\right) }{\tau_{W}\left(
e_{x}\right) }=\partial_{x}\log\tau_{W}\left( e_{x}\right) \text{,}
\end{align*}
which yields (\ref{24}).\medskip
\end{proof}

Since $\tau_{W}\left( e_{x}\right) $ is an entire function of $x$, we see
that $q_{W}(x)$ is meromorphic on $\mathbb{C}$, and it has poles of degree $%
2 $ on $\left\{ x\in\mathbb{C}\text{; }\tau_{W}\left( e_{x}\right)
=0\right\} $. Proposition\ref{p3} and Lemma\ref{l4} are rearrangement of the
corresponding results obtained by \cite{s-w}.

The formula (\ref{21}) defines a map from $W\in Gr^{(2)}$ to a space of
potentials, however this map is not injective. Later we will see a quantity
of $W$, which will be called as $m$-function, determines $q_{W}$.

\section{$m$-function}

In this section we define a crucial quantity of $W\in Gr^{(2)}$ which
determines $q_{W}$. Let $f_{W}(x,z)$ be the function introduced in Lemma\ref%
{l4} and define the $m$\textit{-}\emph{function} for $W\in Gr^{\left(
2\right) }$ by%
\begin{equation}
m_{W}\left( z\right) =-\dfrac{f_{W}^{\prime}(0,z)}{f_{W}(0,z)}\text{.}
\label{25}
\end{equation}
$m_{W}$ can be described by the elements of the characteristic matrix as
follows. Setting%
\begin{equation*}
f_{W}(x,z)=e^{-xz}\left( 1+\dfrac{a_{1}(x)}{z}+\dfrac{a_{2}(x)}{z^{2}}%
+\cdots\right) \in W\text{,}
\end{equation*}
we have%
\begin{align*}
f_{W}(0,z) & =1+\dfrac{a_{1}(0)}{z}+\dfrac{a_{2}(0)}{z^{2}}+\cdots
=1+\varphi_{W}\left( z\right) \\
f_{W}^{\prime}(0,z) & =-z-a_{1}(0)+\dfrac{a_{1}^{\prime}(0)-a_{2}(0)}{z}%
+\cdots \\
& =-z-\psi_{W}\left( z\right) -a_{1}(0)\left( 1+\varphi_{W}\left( z\right)
\right) \text{,}
\end{align*}
due to $f_{W}(0,z)$, $f_{W}^{\prime}(0,z)\in W$, hence%
\begin{equation}
m_{W}\left( z\right) =\frac{z+\psi_{W}\left( z\right) }{1+\varphi _{W}\left(
z\right) }+a_{1}\left( W\right) \text{,}  \label{26}
\end{equation}
where we have defined%
\begin{equation}
a_{1}\left( W\right) =a_{1}(0)=\lim_{z\rightarrow\infty}z\varphi_{W}\left(
z\right) \text{.}  \label{27}
\end{equation}
One of the importance of $m$-functions lies in its close relationship with $%
\tau$-functions. Namely, one can decompose $\tau_{W}\left( g\right) $ into
the two factors, one of which is a group homomorphism from $\mathit{\Gamma}$
to $\mathbb{C}$ and the other part depends only on $m_{W}$.

Since%
\begin{equation*}
1+\varphi_{W}\left( z\right) =1+\dfrac{a_{1}(0)}{z}+\dfrac{a_{2}(0)}{z^{2}}%
+\cdots\text{,}
\end{equation*}
there exists $r_{W}>r$ such that $1+\varphi_{W}\left( z\right) \neq0$ on $%
\left\{ \left\vert z\right\vert \geq r_{W}\right\} $. Let%
\begin{equation*}
\log\left( 1+\varphi_{W}\left( z\right) \right)
=b_{1}z^{-1}+b_{2}z^{-2}+\cdots\text{.}
\end{equation*}
Set%
\begin{equation}
\mathit{\Gamma}_{W}=\left\{ g=e^{h}\text{; }h\text{ holomorphic on }\left\{
\left\vert z\right\vert <r_{W}+\epsilon\right\} \text{ for some }%
\epsilon>0\right\} \text{,}  \label{55}
\end{equation}
and define%
\begin{equation}
\rho_{W}\left( e^{h}\right) =\exp\left(
\sum_{k=1}^{\infty}kb_{k}h_{k}\right) =\exp\left( \dfrac{1}{2\pi i}%
\int_{\left\vert z\right\vert =r_{W}}h^{\prime}\left( z\right) \log\left(
1+\varphi_{W}\left( z\right) \right) \right) dz  \label{28}
\end{equation}
\ for $h(z)=\sum_{k=1}^{\infty}h_{k}z^{k}$. Then, (\ref{28}) is convergent
for $g\in\mathit{\Gamma}_{W}$. If $g=q_{\zeta}$ with $\left\vert
\zeta\right\vert >r_{W}$, then, in view of $\log
q_{\zeta}(z)=\sum_{k=1}^{\infty}\zeta ^{-k}z^{k}/k$%
\begin{equation}
\rho_{W}\left( q_{\zeta}\right) =\exp\left( \sum_{k=1}^{\infty}\zeta
^{-k}b_{k}\right) =1+\varphi_{W}\left( \zeta\right)  \label{29}
\end{equation}
holds. For a holomorphic function $m$ on $\left\{ \left\vert z\right\vert
>r\right\} $ and $\zeta\in\left\{ \left\vert z\right\vert >r\right\} $ define%
\begin{equation*}
\left( d_{\zeta}m\right) \left( z\right) =\dfrac{z^{2}-\zeta^{2}}{m\left(
z\right) -m\left( \zeta\right) }-m\left( \zeta\right) \text{.}
\end{equation*}
\ \ \ One can see easily that $d_{\zeta_{1}}d_{\zeta_{2}}=d_{\zeta_{2}}d_{%
\zeta_{1}}$. Then, (\ref{50}) in the Appendix implies that $m_{q_{\zeta
_{1}}q_{\zeta_{2}}\cdots q_{\zeta_{n}}W}\left( z\right) $ for $\left\vert
\zeta_{k}\right\vert >r$, $k=1,2,\cdots,n$ is generated from $m_{W}$ by%
\begin{equation*}
m_{q_{\zeta_{1}}q_{\zeta_{2}}\cdots q_{\zeta_{n}}W}\left( z\right) =\left(
d_{\zeta_{1}}d_{\zeta_{2}}\cdots d_{\zeta_{n}}m_{W}\right) \left( z\right) 
\text{.}
\end{equation*}
For $g=q_{\zeta_{1}}q_{\zeta_{2}}\cdots q_{\zeta_{n}}$ with $\left\vert
\zeta_{k}\right\vert >r$, $k=1,2,\cdots,n$ define $\tau_{m}\left( g\right) $
inductively\ by%
\begin{equation}
\left\{ 
\begin{array}{l}
\tau_{m}\left( q_{\zeta_{1}}\right) =1 \\ 
\tau_{m}\left( q_{\zeta_{1}}q_{\zeta_{2}}\cdots q_{\zeta_{n}}\right)
\tau_{m}\left( q_{\zeta_{1}}q_{\zeta_{2}}\cdots q_{\zeta_{n-1}}\right) ^{-1}
\\ 
=\dprod \limits_{k=1}^{n-1}\dfrac{\left( d_{\zeta_{1}}d_{\zeta_{2}}\cdots
d_{\zeta_{n-k-1}}m\right) \left( \zeta_{n}\right) -\left(
d_{\zeta_{1}}d_{\zeta_{2}}\cdots d_{\zeta_{n-k-1}}m\right) \left(
\zeta_{n-k}\right) }{\zeta_{n}-\zeta _{n-k}}%
\end{array}
\right. \text{.}  \label{30}
\end{equation}

\begin{proposition}
\label{p4}$\rho_{W}$ satisfies%
\begin{equation}
\rho_{W}(g_{1}g_{2})=\rho_{W}(g_{1})\rho_{W}(g_{2})\text{,}  \label{31}
\end{equation}
for any $g_{1}$, $g_{2}\in\mathit{\Gamma}_{W}$. Moreover, $\tau_{m}\left(
g\right) $ is extendable to $\mathit{\Gamma}_{W}$ if $m=m_{W}$ so that%
\begin{equation}
\tau_{W}(g)=\rho_{W}(g)\tau_{m_{W}}\left( g\right)  \label{32}
\end{equation}
holds, and $\tau_{m_{W}}\left( g\right) $ depends on $W$ only through $m_{W} 
$, namely if $m_{W_{1}}=m_{W_{2}}$ for $W_{1}$, $W_{2}\in Gr^{(2)}$, then $%
\tau_{m_{W_{1}}}\left( g\right) =\tau_{m_{W_{2}}}\left( g\right) $ holds.
\end{proposition}

\begin{proof}
(\ref{31}) follows easily from the definition. We show (\ref{32}) for $%
g=q_{\zeta_{1}}q_{\zeta_{2}}\cdots q_{\zeta_{n}}$ with $\left\vert \zeta
_{k}\right\vert >r_{W}$, $k=1,2,\cdots,n$. If $n=2$, then, from (\ref{51})
in the Appendix%
\begin{equation*}
\tau_{W}\left( q_{\zeta_{1}}q_{\zeta_{2}}\right) =\left( 1+\varphi
_{W}\left( \zeta_{1}\right) \right) \left( 1+\varphi_{W}\left( \zeta
_{2}\right) \right) \dfrac{m_{W}\left( \zeta_{1}\right) -m_{W}\left(
\zeta_{2}\right) }{\zeta_{1}-\zeta_{2}}\text{,}
\end{equation*}
so that in this case%
\begin{equation*}
\dfrac{\tau_{W}\left( q_{\zeta_{1}}q_{\zeta_{2}}\right) }{%
\rho_{W}(q_{\zeta_{1}}q_{\zeta_{2}})}=\dfrac{m_{W}\left( \zeta_{1}\right)
-m_{W}\left( \zeta_{2}\right) }{\zeta_{1}-\zeta_{2}}=\tau_{m_{W}}\left(
q_{\zeta_{1}}q_{\zeta_{2}}\right) \text{.}
\end{equation*}
Assume%
\begin{equation*}
\tau_{W}\left( q_{\zeta_{1}}q_{\zeta_{2}}\cdots q_{\zeta_{n-1}}\right)
=\rho_{W}(q_{\zeta_{1}}q_{\zeta_{2}}\cdots
q_{\zeta_{n-1}})\tau_{m_{W}}\left( q_{\zeta_{1}}q_{\zeta_{2}}\cdots
q_{\zeta_{n-1}}\right) \text{.}
\end{equation*}
Then, from (\ref{19}) and (\ref{20})%
\begin{align*}
& \tau_{W}\left( q_{\zeta_{1}}q_{\zeta_{2}}\cdots
q_{\zeta_{n-1}}q_{\zeta_{n}}\right) \\
& =\tau_{W}\left( q_{\zeta_{1}}q_{\zeta_{2}}\cdots q_{\zeta_{n-1}}\right)
\tau_{q_{\zeta_{1}}q_{\zeta_{2}}\cdots q_{\zeta_{n-1}}W}\left(
q_{\zeta_{n}}\right) \\
& =\rho_{W}(q_{\zeta_{1}}q_{\zeta_{2}}\cdots
q_{\zeta_{n-1}})\tau_{m_{W}}\left( q_{\zeta_{1}}q_{\zeta_{2}}\cdots
q_{\zeta_{n-1}}\right) \left( 1+\varphi_{q_{\zeta_{1}}q_{\zeta_{2}}\cdots
q_{\zeta_{n-1}}W}\left( \zeta _{n}\right) \right)
\end{align*}
follows. On the other hand, iterated use of (\ref{47}) shows%
\begin{align*}
& 1+\varphi_{q_{\zeta_{1}}q_{\zeta_{2}}\cdots q_{\zeta_{n-1}}W}\left(
\zeta_{n}\right) \\
& =\left( 1+\varphi_{W}\left( \zeta_{n}\right) \right) \dprod _{k=1}^{n-1}%
\dfrac{m_{q_{\zeta_{1}}q_{\zeta_{2}}\cdots q_{\zeta_{n-k-1}}W}\left(
\zeta_{n}\right) -m_{q_{\zeta_{1}}q_{\zeta_{2}}\cdots
q_{\zeta_{n-k-1}}W}\left( \zeta_{n-k}\right) }{\zeta_{n}-\zeta_{n-k}}\text{.}
\end{align*}
Thus, for $g=q_{\zeta_{1}}q_{\zeta_{2}}\cdots q_{\zeta_{n}}$ (\ref{32}) is
valid. At each step $q_{\zeta_{1}}q_{\zeta_{2}}\cdots q_{\zeta_{n-k-1}}W\in
Gr^{\left( 2\right) }$ should be examined. However, the final identity (\ref%
{32}) implies that we have only to take some limit if necessary. For general 
$g=e^{h}\in\mathit{\Gamma}_{W}$ let $h_{m}(z)=\sum_{k=1}^{m}h_{k}z^{k}$, and 
$\left\{ \zeta_{k}^{\left( n\right) }\right\} _{1\leq k\leq m}$ be the all
zeroes of $1-h_{m}(z)/n$. One can assume $\left\vert \zeta _{k}^{\left(
n\right) }\right\vert >r_{W}$, because $1-h_{m}(z)/n\rightarrow 1$ as $%
n\rightarrow\infty$. Since%
\begin{equation*}
g_{n}^{\left( m\right) }\left( z\right) \equiv\left( 1-\dfrac{h_{m}(z)}{n}%
\right) ^{-n}=\left( q_{\zeta_{1}^{\left( n\right) }}\left( z\right)
q_{\zeta_{2}^{\left( n\right) }}\left( z\right) \cdots q_{\zeta_{m}^{\left(
n\right) }}\left( z\right) \right) ^{n}\text{,}
\end{equation*}
the identity (\ref{32}) is valid for $g_{n}^{\left( m\right) }\in\mathit{%
\Gamma}_{W}$. Then, the continuity of $\tau_{W}$ and $\rho_{W}$ show that $%
\tau_{m_{W}}$ is extendable by letting $n\rightarrow\infty$ and $%
m\rightarrow\infty$.\medskip
\end{proof}

Corollary below shows that the non-vanishing property of $\tau_{W}(g)$ on $%
\mathit{\Gamma}$ is determined by the $m$-function.

\begin{corollary}
\label{c2}For $W_{1}$, $W_{2}\in Gr^{\left( 2\right) }$ assume $%
m_{W_{1}}\left( z\right) =m_{W_{2}}\left( z\right) $. Then, for $g\in 
\mathit{\Gamma}$ it holds that $\tau_{W_{1}}\left( g\right) \neq0$ is valid
if and only if $\tau_{W_{2}}\left( g\right) \neq0$.
\end{corollary}

\begin{proof}
Assume $\tau_{W_{1}}\left( g\right) \neq0$. Then, $\tau_{W_{1}}\left(
e^{h_{n}}\right) \neq0$ for every sufficiently large $n$, where $%
h_{n}(z)=\sum_{k=1}^{n}h_{k}z^{k}$, and Proposition\ref{p4} implies%
\begin{equation*}
\tau_{W_{2}}\left( \widetilde{g}e^{h_{n}}\right) =\dfrac{\rho_{W_{2}}\left( 
\widetilde{g}\right) \rho_{W_{2}}\left( e^{h_{n}}\right) }{%
\rho_{W_{1}}\left( \widetilde{g}\right) \rho_{W_{1}}\left( e^{h_{n}}\right) }%
\tau_{W_{1}}\left( \widetilde{g}e^{h_{n}}\right)
\end{equation*}
for any $\widetilde{g}\in\mathit{\Gamma}_{W_{1}}\cap\mathit{\Gamma}_{W_{2}}$%
. Since $\tau_{W_{1}}\left( \widetilde{g}e^{h_{n}}\right) $, $%
\tau_{W_{2}}\left( \widetilde{g}e^{h_{n}}\right)
\rightarrow\tau_{W_{1}}\left( \widetilde{g}g\right) $, $\tau_{W_{2}}\left( 
\widetilde{g}g\right) $ respectively, there exists a $c\in\mathbb{C}$ such
that $\rho_{W_{2}}\left( e^{h_{n}}\right) /\rho_{W_{1}}\left(
e^{h_{n}}\right) \rightarrow c$, and%
\begin{equation*}
\tau_{W_{2}}\left( \widetilde{g}g\right) =c\dfrac{\rho_{W_{2}}\left(
g_{1}\right) }{\rho_{W_{1}}\left( g_{1}\right) }\tau_{W_{1}}\left( 
\widetilde{g}g\right)
\end{equation*}
holds. Suppose $c=0$. Then, $\tau_{W_{2}}\left( \widetilde{g}g\right) =0$
for any $\widetilde{g}\in\mathit{\Gamma}_{W_{1}}\cap\mathit{\Gamma}_{W_{2}}$%
, which contradicts $\tau_{W_{2}}\left( 1\right) =1$, if we choose $%
\widetilde{g}=e^{-h_{n}}$. Therefore, we have $c\neq0$, which shows $%
\tau_{W_{2}}\left( g\right) \neq0$.
\end{proof}

\begin{corollary}
\label{c3}For $W\in Gr^{\left( 2\right) }$ it holds that%
\begin{equation*}
q_{W}(x)=-2\partial_{x}^{2}\log\tau_{m_{W}}\left( e_{x}\right) \text{.}
\end{equation*}
\end{corollary}

\begin{proof}
Since $\rho_{W}\left( e_{x}\right) =e^{xb_{1}}$, Proposition\ref{p4}
completes the proof.$\medskip$
\end{proof}

To show the continuity of $m_{gW}$ with respect to $g$ we need a
representation of $m_{W}$ by the $\tau$-functions.

\begin{lemma}
\label{l5}For $W\in Gr^{\left( 2\right) }$ we have%
\begin{equation*}
m_{W}\left( \zeta\right) =\zeta+\dfrac{1}{2\pi i}\int_{\left\vert
\omega\right\vert =R}\left( \tau_{W}\left( q_{\omega}\right) -\dfrac {%
\tau_{W}\left( q_{\zeta}q_{\omega}\right) }{\tau_{W}\left( q_{\zeta }\right) 
}\right) d\omega
\end{equation*}
for any $R>r$ and $\zeta$ such that $\tau_{W}\left( q_{\zeta}\right)
=1+\varphi_{W}\left( \zeta\right) \neq0$.
\end{lemma}

\begin{proof}
(\ref{49}) reads%
\begin{equation*}
m_{W}\left( \zeta\right) =\zeta+a_{1}(W)-a_{1}(q_{\zeta}W)
\end{equation*}
with $a_{1}(W)$ the first coefficient of the expansion for $\varphi_{W}$,
hence%
\begin{equation*}
a_{1}(W)=\dfrac{1}{2\pi i}\int_{\left\vert \omega\right\vert =R}\varphi
_{W}\left( \omega\right) d\omega\text{.}
\end{equation*}
Since $\varphi_{W}\left( \zeta\right) =\tau_{W}\left( q_{\zeta}\right) -1 $,
we have%
\begin{align*}
m_{W}\left( \zeta\right) & =\zeta+\dfrac{1}{2\pi i}\int_{\left\vert
\omega\right\vert =R}\left( \tau_{W}\left( q_{\omega}\right) -\tau
_{q_{\zeta}W}\left( q_{\omega}\right) \right) d\omega \\
& =\zeta+\dfrac{1}{2\pi i}\int_{\left\vert \omega\right\vert =R}\left(
\tau_{W}\left( q_{\omega}\right) -\dfrac{\tau_{W}\left( q_{\zeta}q_{\omega
}\right) }{\tau_{W}\left( q_{\zeta}\right) }\right) d\omega\text{.}
\end{align*}
\end{proof}

\begin{proposition}
\label{p5}Suppose $g_{n}$, $g\in\mathit{\Gamma}$ and $\tau_{W}\left(
g_{n}\right) \neq0$, $\tau_{W}\left( g\right) \neq0$. If $g_{n}\rightarrow g$
in $H^{2}\left( \partial\mathbb{D}_{r}\right) $, then $m_{g_{n}W}\left(
\zeta\right) \rightarrow m_{gW}\left( \zeta\right) $ for any $\zeta$ such
that $\left\vert \zeta\right\vert >r$, $\tau_{W}\left( q_{\zeta}g\right)
\neq0$.
\end{proposition}

\begin{proof}
The integral representation of Lemma\ref{l5} completes the proof.\smallskip
\end{proof}

The following proposition says that $q_{W}$ is determined by $m_{W}$.

\begin{proposition}
\label{p6}For $W_{1}$, $W_{2}\in Gr^{\left( 2\right) }$ we have $%
q_{W_{1}}=q_{W_{2}}$ if and only if $m_{W_{1}}=m_{W_{2}}$. Moreover, assume $%
m_{W_{1}}=m_{W_{2}}$. Then, it holds that%
\begin{equation}
m_{gW_{1}}\left( z\right) =m_{gW_{2}}\left( z\right)  \label{33}
\end{equation}
for any $g\in\mathit{\Gamma}$ such that $\tau_{W_{1}}\left( g\right) \neq0$
(hence $\tau_{W_{2}}\left( g\right) \neq0$).
\end{proposition}

\begin{proof}
Set%
\begin{equation*}
\dfrac{f_{W}(x,z)}{f_{W}(0,z)}=e^{-xz}\left( 1+\dfrac{\widetilde{a}_{1}(x)}{z%
}+\dfrac{\widetilde{a}_{2}(x)}{z^{2}}+\cdots\right) \text{.}
\end{equation*}
Then, (\ref{2}) implies%
\begin{equation*}
\left\{ 
\begin{array}{l}
q_{W}(x)=-2\widetilde{a}_{1}^{\prime}(x)\text{ \ (since }\widetilde{a}%
_{1}(x)=a_{1}(x)-a_{1}(0)\text{)} \\ 
2\widetilde{a}_{k+1}^{\prime}\left( x\right) -\widetilde{a}%
_{k}^{\prime\prime}\left( x\right) -2\widetilde{a}_{1}^{\prime}\left(
x\right) \widetilde{a}_{k}\left( x\right) =0\text{,}\ \ \ k=1,2,\cdots%
\end{array}
\right. \text{.}
\end{equation*}
Due to $\widetilde{a}_{k}\left( 0\right) =0$ for any $k=1,2,\cdots$ we see
that $q_{W}$ determines $f_{W}(x,z)/f_{W}(0,z)$. Keeping this in mind,
suppose $q_{W_{1}}=q_{W_{2}}$. Then%
\begin{equation*}
\dfrac{f_{W_{1}}(x,z)}{f_{W_{1}}(0,z)}=\dfrac{f_{W_{2}}(x,z)}{f_{W_{2}}(0,z)}
\end{equation*}
holds, which implies%
\begin{equation*}
\dfrac{f_{W_{1}}^{\prime}(0,z)}{f_{W_{1}}(0,z)}=\dfrac{f_{W_{2}}^{\prime
}(0,z)}{f_{W_{2}}(0,z)}
\end{equation*}
and $m_{W_{1}}\left( z\right) =m_{W_{2}}\left( z\right) $ follows.
Conversely, if $m_{W_{1}}\left( z\right) =m_{W_{2}}\left( z\right) $, then
Corollary\ref{c3} shows $q_{W_{1}}=q_{W_{2}}$.\newline
To show the identity (\ref{33}) assume $m_{W_{1}}\left( z\right)
=m_{W_{2}}\left( z\right) $. Then, (\ref{50}) shows that $%
m_{gW}=d_{\zeta_{1}}d_{\zeta_{2}}\cdots d_{\zeta_{n}}m_{W}$ for $%
g=q_{\zeta_{1}}q_{\zeta_{2}}\cdots q_{\zeta_{n}}$, hence Proposition\ref{p5}
shows $m_{gW_{1}}(z)=m_{gW_{2}}(z)$ for general $g\in\mathit{\Gamma}$.
\end{proof}

\section{KdV flow}

Let $\mathit{\Gamma}$ be a commutative group and $\mathcal{Q}$ be a set.
Suppose there exists a set of maps $\left\{ K(g)\right\} _{g\in \mathit{%
\Gamma}}$ on $\mathcal{Q}$ satisfying $K(g_{1}g_{2})=K(g_{1})K(g_{2})$ for
any $g_{1}$, $g_{2}\in\mathit{\Gamma}$, we call $\left\{ K(g)\right\} _{g\in%
\mathit{\Gamma}}$ as a flow on $\mathcal{Q}$. The purpose of this section is
to construct such a flow on a certain set of potentials $\mathcal{Q}$ and a
subgroup of the previous $\mathit{\Gamma}$.

\subsection{$m$-function and Weyl functions}

In the last section we defined the $m$-function $m_{W}$ for any $W\in
Gr^{\left( 2\right) }$. On the other hand, for $W\in Gr^{\left( 2\right) }$
a potential $q_{W}$ was introduced by (\ref{21}), and if $q_{w}$ takes real
values, then one can define the Weyl functions $m_{\pm}$. If the
Baker-Akhiezer function $f_{W}$ belongs to $L^{2}\left( \mathbb{R}%
_{+}\right) $, then we have $m(z)=-m_{+}\left( -z^{2}\right) $. In this
subsection we investigate this identity by imposing an additional condition
on $W\in Gr^{\left( 2\right) }$, namely $\tau_{W}\left( g\right) \neq0$ for
any real $g\in\mathit{\Gamma}$.

Recall $\overline{f}\left( z\right) =\overline{f\left( \overline{z}\right) }$
for $f\in H=L^{2}\left( \partial\mathbb{D}_{r}\right) $, and set $\overline{W%
}=\left\{ f\in H\text{; \ }\overline{f}\in W\right\} $ for $W\in Gr^{\left(
2\right) }$. Then, clearly $\overline{W}\in Gr^{\left( 2\right) }$ holds and
an identity $\mathit{\Pi}_{\overline{W}}=\overline{\mathit{\Pi}}_{W}$ is
straightforward. $W\in Gr^{\left( 2\right) }$ is called \textbf{real} if $W=%
\overline{W}$, and this is the case if and only if $\mathit{\Pi}_{W}=%
\overline{\mathit{\Pi}}_{W}$ is valid. Define%
\begin{equation*}
\left\{ 
\begin{array}{l}
\mathit{\Gamma}_{\func{real}}=\left\{ g\in\mathit{\Gamma};\text{ \ }g=%
\overline{g}\right\} \medskip \\ 
Gr_{\func{real}}^{\left( 2\right) }=\left\{ W\in Gr^{\left( 2\right) };\text{
\ }W=\overline{W}\right\} \smallskip \\ 
Gr_{+}^{\left( 2\right) }=\left\{ W\in Gr_{\func{real}}^{\left( 2\right) };%
\text{ \ }\tau_{W}\left( g\right) \geq0\text{ for any }g\in\mathit{\Gamma}_{%
\func{real}}\right\}%
\end{array}
\right. \text{.}
\end{equation*}
If $W\in Gr_{\func{real}}^{\left( 2\right) }$ and $g\in \mathit{\Gamma}_{%
\func{real}}$, then $\tau_{W}\left( g\right) \in\mathbb{R}$. Recall $%
q_{\zeta}\left( z\right) =\left( 1-z/\zeta\right) ^{-1}$ and define a dual
object%
\begin{equation*}
p_{\zeta}\left( z\right) =1+z/\zeta=q_{-\zeta}\left( z\right) ^{-1}\text{.}
\end{equation*}

\begin{lemma}
\label{l6}The followings are valid.\newline
(i) \ If $\tau_{W}\left( g\right) \geq0$ holds for any $g$ of a form $%
g=\prod_{k=1}^{n}q_{\zeta_{k}}q_{\overline{\zeta}_{k}}$ with $%
\zeta_{k}\in\left\{ \left\vert z\right\vert >r\right\} $ and $\func{Im}%
\zeta_{k}\neq0$, then $W\in Gr_{+}^{\left( 2\right) }$ is valid.\newline
(ii) Suppose $\tau_{W}\left( q_{\zeta}q_{\overline{\zeta}}\right) \geq0$ for
any $\zeta\in\left\{ \left\vert z\right\vert >r\right\} $ for a $W\in Gr_{%
\func{real}}^{\left( 2\right) }$. Then $\tau_{W}\left( q_{\zeta}q_{\overline{%
\zeta}}\right) >0$ holds for any $\zeta\in\left\{ \left\vert z\right\vert
>r\right\} $.\newline
(iii) Assume $W\in Gr_{+}^{\left( 2\right) }$. Then, $\tau_{W}\left(
g\right) >0$ holds for any $g\in\mathit{\Gamma }_{\func{real}}$ such that $%
g(z)=\prod_{k=1}^{n}q_{\zeta_{k}}q_{\overline{\zeta}_{k}}$ or $%
g(z)=\prod_{k=1}^{n}p_{\zeta_{k}}p_{\overline{\zeta}_{k}}$ with $%
\zeta_{k}\in\left\{ \left\vert z\right\vert >r\right\} $ and $\func{Im}%
\zeta_{k}\neq0$.
\end{lemma}

\begin{proof}
To show (i) let $g=e^{h}\in\mathit{\Gamma}_{\func{real}}$, $%
h(z)=\sum_{k=1}^{\infty}h_{k}z^{k}$ and $h_{n}(z)=\sum_{k=1}^{2n}h_{k}z^{k}$%
. Set $g_{n}(z)=\left( 1-h_{n}(z)/n\right) ^{-n}$. Then, $g(z)=\lim
_{n\rightarrow\infty}g_{n}(z)$ holds. We have only to show $\tau_{W}\left(
g_{n}\right) \geq0$ for sufficiently large $n$ due to the continuity of $%
\tau_{W}$. One can assume $h_{2n}\neq0$ and $1-h_{n}\left( z\right) /n$ has
no real zeros, since, otherwise we have only to deform slightly $h$. Then,
there exist $\zeta_{k}\in\left\{ \left\vert z\right\vert >r\right\} $ and $%
\func{Im}\zeta_{k}\neq0$ for $k=1,2,\cdots,n$ such that%
\begin{equation}
g_{n}(z)=\left( 1-\dfrac{h_{n}(z)}{n}\right) ^{-n}=\left( \dprod
\limits_{k=1}^{n}q_{\zeta_{k}}\left( z\right) q_{\overline{\zeta_{k}}}\left(
z\right) \right) ^{n}  \label{34}
\end{equation}
holds, which proves (i).\newline
To prove (ii) note that%
\begin{align*}
\det\mathit{\Pi}_{W}\left( z^{2}\right) & =\left( 1+\varphi_{W,e}\left(
z^{2}\right) \right) \left( 1+\psi_{W,o}\left( z^{2}\right) \right)
-\varphi_{W,o}\left( z^{2}\right) \psi_{W,e}\left( z^{2}\right) \\
& =\dfrac{\left( 1+\varphi_{W}\left( -z\right) \right) \left( z+\psi
_{W}\left( z\right) \right) -\left( 1+\varphi_{W}\left( z\right) \right)
\left( -z+\psi_{W}\left( -z\right) \right) }{2z}
\end{align*}
is valid, hence $1+\varphi_{W}\left( z\right) $ and $z+\psi_{W}\left(
z\right) $ do not vanish simultaneously due to $\det\mathit{\Pi}_{W}\left(
z\right) \neq0$. Suppose%
\begin{equation*}
\tau_{W}\left( q_{\zeta_{1}}q_{\overline{\zeta}_{1}}\right) =0\text{ \ and \ 
}1+\varphi_{W}\left( \zeta_{1}\right) \neq0\text{ \ for }\zeta_{1}\text{
such that }\func{Im}\zeta_{1}>0\text{.}
\end{equation*}
From (\ref{51})%
\begin{equation*}
\tau_{W}\left( q_{\zeta}q_{\overline{\zeta}}\right) =\left\vert
1+\varphi_{W}\left( \zeta\right) \right\vert ^{2}\dfrac{\func{Im}m_{W}\left(
\zeta\right) }{\func{Im}\zeta}
\end{equation*}
holds, hence\ $\func{Im}m_{W}\left( \zeta\right) \geq0$ is valid\ if \ $%
\func{Im}\zeta>0$. Therefore, if $\tau_{W}\left( q_{\zeta_{1}}q_{\overline{%
\zeta_{1}}}\right) =0$, then $\func{Im}m_{W}\left( \zeta\right) =0$
identically due to the fact that $\func{Im}m_{W}\left( \zeta\right) $ is
harmonic and $\func{Im}m_{W}\left( \zeta_{1}\right) =0$. However, as $%
\zeta\rightarrow\infty$, $\tau_{W}\left( q_{\zeta}q_{\overline{\zeta}%
}\right) \rightarrow1$ holds, which leads us to contradiction. The case $%
\zeta_{1}+\psi_{W}\left( \zeta_{1}\right) \neq0$ can be treated similarly.
Thus $\tau_{W}\left( q_{\zeta_{1}}q_{\overline {\zeta}_{1}}\right) >0$
should hold for any $\zeta_{1}$.\newline
We prove (iii) by induction. For $n=1$ (ii) implies the strict positivity of 
$\tau _{W}\left( q_{\zeta}q_{\overline{\zeta}}\right) $. Assume $%
\tau_{W}\left( g_{1}\right) >0$ is valid for $g_{1}=\prod_{k=1}^{n-1}q_{%
\zeta_{k}}q_{\overline{\zeta}_{k}}$. For any $h\in\mathit{\Gamma}_{\func{real%
}} $%
\begin{equation*}
\tau_{g_{1}W}\left( h\right) =\frac{\tau_{W}\left( g_{1}h\right) }{%
\tau_{W}\left( g_{1}\right) }\geq0
\end{equation*}
holds, hence $g_{1}W\in Gr_{2}^{+}$ and the argument above shows $\tau
_{g_{1}W}\left( q_{\zeta_{n}}q_{\overline{\zeta}_{n}}\right) >0$. Now, for $%
g=g_{1}q_{\zeta_{n}}q_{\overline{\zeta}_{n}}$ an identity%
\begin{equation*}
\tau_{W}\left( g\right) =\tau_{g_{1}W}\left( q_{\zeta_{n}}q_{\overline {\zeta%
}_{n}}\right) \tau_{W}\left( g_{1}\right)
\end{equation*}
shows $\tau_{W}\left( g\right) >0$. For $g(z)=\prod_{k=1}^{n}p_{\zeta_{k}}p_{%
\overline{\zeta}_{k}}$ the identity (\ref{52}) in the Appendix implies%
\begin{equation*}
\tau_{W}\left( p_{\zeta_{1}}p_{\zeta_{2}}\cdots p_{\zeta_{n}}\right) =\left(
\tau_{W}\left( r_{\zeta_{1}}\right) \tau_{W}\left( r_{\zeta_{2}}\right)
\cdots\tau_{W}\left( r_{\zeta_{n}}\right) \right) ^{-1}\tau _{W}\left(
q_{\zeta_{1}}q_{\zeta_{2}}\cdots q_{\zeta_{n}}\right) \text{.}
\end{equation*}
Since $\tau_{W}\left( r_{\zeta}\right) \neq0$, the proof is complete.$%
\medskip$
\end{proof}

\textit{Proof of Theorem 1.} Assume $W\in Gr_{\func{real}}^{\left( 2\right)
} $ satisfies $gW\in Gr^{\left( 2\right) }$ for any $g\in \mathit{\Gamma}_{%
\func{real}}$. Then, $\tau_{W}\left( q_{\overline {\zeta}}q_{\zeta}\right)
\neq0$ holds for any $\zeta\in\left\{ \left\vert z\right\vert >r\right\} $.
Since $\tau_{W}\left( q_{\overline{\zeta}}q_{\zeta}\right) $ takes real
values and approaches to $1$ as $\zeta \rightarrow\infty$, we have $%
\tau_{W}\left( q_{\overline{\zeta}}q_{\zeta }\right) >0$. Inductively one
can see $\tau_{W}\left( g\right) >0$ for any $g=\prod_{k=1}^{n}q_{%
\zeta_{k}}q_{\overline{\zeta}_{k}}$ with $\zeta_{k}\in\left\{ \left\vert
z\right\vert >r\right\} $, which shows $W\in Gr_{+}^{\left( 2\right) }$ due
to (i) of Lemma\ref{l6}.

To show the converse direction let $g$ be in $\mathit{\Gamma}_{\func{real}}$
and $g_{n}$ be the function defined in (\ref{34}) by replacing $g$ with $%
g^{-1}$, hence $g_{n}\rightarrow g^{-1}$ in this case. Then, Lemma\ref{l6}
shows $\tau_{W}\left( g_{n}\right) >0$. On the other hand, since $\tau_{W}$
is continuous and $\tau_{W}\left( 1\right) =1$, we see%
\begin{equation*}
\tau_{W}\left( gg_{n}\right) >0
\end{equation*}
for sufficiently large $n$. Since $gg_{n}W\in Gr_{+}^{\left( 2\right) }$ is
valid (see the argument in the proof of (iii) of Lemma \ref{l6}), from Lemma%
\ref{l6} it follows that $\tau_{gg_{n}W}\left( p\right) >0$ for any $p$ of
the form $\prod_{k=1}^{n}p_{\zeta_{k}^{\prime}}p_{\overline{\zeta}%
_{k}^{\prime}}$. Therefore%
\begin{equation*}
\tau_{W}\left( gg_{n}p\right) =\tau_{gg_{n}W}\left( p\right) \tau _{W}\left(
gg_{n}\right) >0
\end{equation*}
is valid. Now, taking $p=g_{n}^{-1}$, we see $\tau_{W}\left( g\right) >0$,
which shows $gW\in Gr_{+}^{\left( 2\right) }$. This completes the proof. $%
\blacksquare$\medskip

\begin{corollary}
\label{c4}Suppose $W\in Gr_{+}^{\left( 2\right) }$. Then, the followings are
valid.\newline
(i) \ $a\left( 1+\varphi_{W}(z)\right) +b\left( z+\psi _{W}(z)\right) $ has
no zeros in $\left\{ \left\vert z\right\vert >r\right\} \cap\left( \mathbb{C}%
\backslash\mathbb{R}\right) $ for any $a,b\in\mathbb{R}$ such that $%
\left\vert a\right\vert +\left\vert b\right\vert \neq0$. Moreover, $%
1+\varphi_{W}(x)>0$ holds for any $x\in\mathbb{R}$ such that $\left\vert
x\right\vert >r$.\newline
(ii) $m_{W}(z)$ is holomorphic on $\left\{ \left\vert z\right\vert
>r\right\} $ and has no zeros in $\left\{ \left\vert z\right\vert >r\right\}
\cap\left( \mathbb{C}\backslash \mathbb{R}\right) $. Moreover, $m_{W}$
satisfies%
\begin{equation*}
\left\{ 
\begin{array}{l}
\dfrac{\func{Im}m_{W}(z)}{\func{Im}z}>0\text{ \ for }z\in\left\{ \left\vert
z\right\vert >r\right\} \text{ with }\func{Im}z\neq0 \\ 
\dfrac{m_{W}(x)-m_{W}(-x)}{2x}>0\text{, \ }m_{W}^{\prime}(x)>0\text{ \ on }%
\mathbb{R}\backslash\left[ -r,r\right] \text{.}%
\end{array}
\right.
\end{equation*}
\end{corollary}

\begin{proof}
For $\zeta\in\left\{ \left\vert z\right\vert >r\right\} \cap\left( \mathbb{C}%
\backslash\mathbb{R}\right) $%
\begin{equation*}
\tau_{W}\left( q_{\zeta}q_{\overline{\zeta}}\right) =\dfrac{1}{\zeta-%
\overline{\zeta}}\left( 
\begin{array}{c}
\left( \zeta+\psi_{W}\left( \zeta\right) \right) \left( \overline {%
1+\varphi_{W}\left( \zeta\right) +b\left( \zeta+\psi_{W}\left( \zeta\right)
\right) }\right) \\ 
-\left( \overline{\zeta+\psi_{W}\left( \zeta\right) }\right) \left(
1+\varphi_{W}\left( \zeta\right) +b\left( \zeta+\psi_{W}\left( \zeta\right)
\right) \right)%
\end{array}
\right)
\end{equation*}
is valid. Hence, if $1+\varphi_{W}\left( \zeta\right) +b\left( \zeta
+\psi_{W}\left( \zeta\right) \right) =0$, then $\tau_{W}\left( q_{\zeta }q_{%
\overline{\zeta}}\right) =0$, which contradicts Theorem\ref{t2}. Similarly
we have $a\left( 1+\varphi_{W}\left( \zeta\right) \right)
+\zeta+\psi_{W}\left( \zeta\right) \neq0$. On the other hand, since $q_{x}\in%
\mathit{\Gamma}_{\func{real}}$ for $x\in\mathbb{R}$ and $\left\vert
x\right\vert >r $, Theorem\ref{t2} implies%
\begin{equation*}
1+\varphi_{W}\left( x\right) =\tau_{W}\left( q_{x}\right) >0\text{,}
\end{equation*}
which shows (i). The first inequality of (ii) follows from%
\begin{equation*}
\tau_{W}\left( q_{\zeta}q_{\overline{\zeta}}\right) =\left\vert
1+\varphi_{W}(\zeta)\right\vert ^{2}\func{Im}m_{W}\left( \zeta\right) /\func{%
Im}\zeta\text{.}
\end{equation*}
The second two inequalities are shown by (\ref{53}) and%
\begin{equation*}
\tau_{W}\left( q_{x}^{2}\right) =\left( 1+\varphi_{W}(x)\right) ^{2}m_{W}(x)%
\text{.}
\end{equation*}
\medskip\medskip
\end{proof}

Theorem\ref{t2} shows that the group $\mathit{\Gamma}_{\func{real}}$ acts on 
$Gr_{+}^{\left( 2\right) }$. Corollary\ref{c2} implies that the
non-vanishing property of $\tau_{W}\left( g\right) $ for $g\in \mathit{\Gamma%
}_{\func{real}}$ can be stated only by $m_{W}$, hence the next task is to
find some concrete criterion in terms of $m_{W}$ for $W\in Gr^{\left(
2\right) }$ to be an element of $Gr_{+}^{\left( 2\right) }$.

To proceed further we prepare some results from the spectral theory of one
dimensional Schr\"{o}dinger operators. For a real valued $q\in
L_{loc}^{1}\left( \mathbb{R}\right) $ let $L_{q}$ be a Schr\"{o}dinger
operator defined by%
\begin{equation*}
\left( L_{q}f\right) \left( x\right) =-f^{\prime\prime}(x)+q(x)f(x)\text{,}
\end{equation*}
and for $\lambda\in\mathbb{C}$ consider a time independent Schr\"{o}dinger
equation%
\begin{equation}
\left( L_{q}f\right) \left( x\right) =\lambda f(x)\text{.}  \label{35}
\end{equation}

\begin{lemma}
\label{l7}(see \cite{c-l})There occur two cases on the behavior of solutions
to (\ref{35}) at $+\infty$.\newline
(i) \ Limit circle type: \ $\dim\left\{ f\in L^{2}\left( \mathbb{R}%
_{+}\right) \text{; \ }L_{q}f=\lambda f\right\} =2$ for any $\lambda\in%
\mathbb{C}$.\newline
(ii) Limit point type: \ $\dim\left\{ f\in L^{2}\left( \mathbb{R}_{+}\right) 
\text{; \ }L_{q}f=\lambda f\right\} =1$ for any $\lambda\in\mathbb{C}%
\backslash \mathrm{sp}L_{+}$.
\end{lemma}

The boundary $-\infty$ has also the same classification. If the boundary $%
+\infty$ is of limit point type, the operator $L_{+}$ is uniquely extendable
as a self-adjoint operator in $L^{2}\left( \mathbb{R}_{+}\right) $, where $%
L_{+}$ is the Schr\"{o}dinger operator $L_{q}$ restricted to $L^{2}\left( 
\mathbb{R}_{+}\right) $ with Dirichlet boundary condition at $0$.

\begin{lemma}
\label{l8}Suppose the boundaries $\pm\infty$ are of limit point type. If
there exists a positive solution $f$ to (\ref{35}), then, $\lambda\leq\inf$ 
\textrm{sp}$L_{q}$ holds.
\end{lemma}

\begin{proof}
Although this is widely known in a more general framework, for completeness
sake we give a proof. For a fixed $a>0$ let $\lambda_{0}$ be the minimum
eigenvalue for the operator $L$ restricted to an interval $\left(
-a,a\right) $ with Dirichlet boundary condition at the boundaries $\pm a$,
and $u$ be the eigenfunction. One can assume $u$ takes positive value in $%
\left( -a,a\right) $. An integration by parts shows%
\begin{align*}
\lambda\int_{-a}^{a}f\left( x\right) u\left( x\right) dx & =\int
_{-a}^{a}L_{q}f\left( x\right) u\left( x\right) dx \\
& =f(a)u^{\prime}(a)-f(-a)u^{\prime}(-a)+\lambda_{0}\int_{-a}^{a}f\left(
x\right) u(x)dx\text{.}
\end{align*}
Since $u^{\prime}(-a)>0$, $u^{\prime}(a)<0$, we have $\lambda<\lambda_{0} $,
which leads us to $\lambda\leq\inf$ \textrm{sp}$L_{q}$ by letting $%
a\rightarrow\infty$.\medskip
\end{proof}

Denote by $f_{\pm}\left( x,\lambda\right) $ the solutions to (\ref{35})
belonging to $L^{2}\left( \mathbb{R}_{\pm}\right) $ respectively when the
boundary $\pm\infty$ are of limit point type, and define%
\begin{equation*}
m_{\pm}\left( \lambda\right) =\pm\dfrac{f_{\pm}^{\prime}\left(
0,\lambda\right) }{f_{\pm}\left( 0,\lambda\right) }\text{.}
\end{equation*}
These two functions $m_{\pm}$ are known to be of \emph{Herglotz} (a
holomorphic function $m$ on $\mathbb{C}\backslash\mathbb{R}$ satisfying $%
m\left( z\right) =\overline{m\left( \overline{z}\right) }$ and $\func{Im}%
m(z)>0$ on $\mathbb{C}_{+}$), and called \emph{Weyl functions} (or
Weyl-Titchmarsh function). $m_{\pm}$ are holomorphic on $\mathbb{C}%
\backslash $\textrm{sp}$L_{\pm}$ respectively. The following proposition
identifies $m_{W}$ with the Weyl functions.

\begin{proposition}
\label{p7}Suppose $W\in Gr_{+}^{\left( 2\right) }$. Then, the associated $%
q_{W}$ has no singularities on $\mathbb{R}$ and $q_{W}$ is real valued
there. Assume $\left\{ m_{e_{x}W}(z)\right\} _{x\in\mathbb{R}}$ forms a
normal family on $\mathbb{C}\backslash\overline{\mathbb{D}}_{r}$. Then, $%
\pm\infty$ are of limit point type for the associated $L_{q_{W}}$, and $%
-r^{2}\leq\inf $\textrm{sp}$L_{q_{W}}$. The $m$-function $m_{W}$ is related
to the Weyl functions $m_{\pm}$ of $q_{W}$ by%
\begin{equation}
m_{W}\left( z\right) =\left\{ 
\begin{array}{l}
-m_{+}\left( -z^{2}\right) \text{, \ \ \ for }\func{Re}z>0 \\ 
m_{-}\left( -z^{2}\right) \text{, \ \ \ \ \ for }\func{Re}z<0%
\end{array}
\right. \text{.}  \label{56}
\end{equation}
Consequently, $m_{W}$ known to be holomorphic on $\mathbb{C}\backslash\left( %
\left[ -r,r\right] \cup i\left[ -r,r\right] \right) $ and have a property%
\begin{equation*}
\dfrac{\func{Im}m_{W}\left( z\right) }{\func{Im}z}>0\text{ \ \ on \ }\mathbb{%
C}\backslash\left( \mathbb{R}\cup i\mathbb{R}\right) \text{.}
\end{equation*}
\end{proposition}

\begin{proof}
$q_{W}$ has no singularity on $\mathbb{R}$ due to Theorem\ref{t2}, since $%
e_{x}\in\mathit{\Gamma}_{\func{real}}$ if $x\in\mathbb{R}$. The first
assertion is The Baker-Akhiezer function $f_{W}(x,z)$ satisfies Schr\"{o}%
dinger equation with potential $q_{W}$ and $\lambda=-z^{2}$. On the other
hand, (\ref{22}) implies%
\begin{equation*}
f_{W}(x+y,\zeta)=e^{-xz-y\zeta}\tau_{e_{x+y}W}\left( q_{\zeta}\right)
=e^{-x\zeta-y\zeta}\tau_{e_{y}e_{x}W}\left( q_{\zeta}\right) =e^{-x\zeta
}f_{e_{x}W}(y,\zeta)\text{.}
\end{equation*}
This together with (\ref{25}) yields%
\begin{equation*}
m_{e_{x}W}\left( \zeta\right) =-\dfrac{f_{e_{x}W}^{\prime}(0,\zeta )}{%
f_{e_{x}W}(0,\zeta)}=-\dfrac{f_{W}^{\prime}(x,\zeta)}{f_{W}(x,\zeta )}\text{,%
}
\end{equation*}
from which an identity%
\begin{equation}
f_{W}(x,\zeta)=f_{W}(0,\zeta)\exp\left( -\int_{0}^{x}m_{e_{y}W}\left(
\zeta\right) dy\right)  \label{36}
\end{equation}
follows. Since $W\in Gr_{+}^{\left( 2\right) }$, Theorem\ref{t2} implies $%
e_{y}W\in Gr_{+}^{\left( 2\right) }$. Therefore, Corollary\ref{c4} shows $%
m_{e_{y}W}\left( z\right) $ is holomorphic on $\mathbb{C}\backslash 
\overline{\mathbb{D}}_{r}$.\newline
Set $\phi_{x}(z)=zm_{e_{x}W}(z^{-1})$. Then, $\phi_{x}$ is holomorphic on $%
\mathbb{D}_{r^{-1}}$ satisfying $\phi _{x}(0)=1$, and $\left\{
\phi_{x}(z)\right\} _{x\in\mathbb{R}}$ forms a normal family on $\mathbb{D}%
_{r^{-1}}$. Denote by $z(x)\in\left( -r^{-1},r^{-1}\right) $ a zero of $%
\phi_{x}$ if it exists. If there exists a sequence $\left\{ x_{n}\right\}
_{n\geq1}\subset\mathbb{R}$ such that $z(x_{n})\rightarrow0\in\left(
-r^{-1},r^{-1}\right) $, then, one can assume $\phi_{x_{n}}\rightarrow\phi$,
and $\phi\left( 0\right) =0$, which contradicts $\phi\left( 0\right) =1$.
Therefore, $\left\vert z(x)\right\vert \geq r_{0}^{-1}$ holds uniformly for
some $r_{0}>r$, which means that $m_{e_{x}W}(z)$ has no zero on $\left\vert
z\right\vert >r_{0}$ for any $x\in\mathbb{R}$. For $\zeta=a>r_{0}$ in (\ref%
{36}) we see%
\begin{equation*}
\left\{ 
\begin{array}{l}
f_{W}\left( x,a\right) =f_{W}\left( 0,a\right) \exp\left(
-\int_{0}^{x}m_{e_{y}W}(a)dy\right) \text{ \ \ decreasing} \\ 
f_{W}\left( x,-a\right) =f_{W}\left( 0,-a\right) \exp\left(
-\int_{0}^{x}m_{e_{y}W}(-a)dy\right) \text{ \ increasing}%
\end{array}
\right. \text{,}
\end{equation*}
and Lemma\ref{l7} implies the boundaries $\pm\infty$ are of limit point
type, which shows%
\begin{equation*}
f_{W}\left( x,a\right) =f_{+}\left( x,-a^{2}\right) \text{, \ \ }f_{W}\left(
x,-a\right) =f_{-}\left( x,-a^{2}\right) \text{.}
\end{equation*}
These identities are valid for any $a>r_{0}$, therefore $f_{W}\left(
x,z\right) =f_{+}\left( x,-z^{2}\right) $, \ \ $f_{W}\left( x,-z\right)
=f_{-}\left( x,-z^{2}\right) $ hold for any $z\in\mathbb{C}\backslash 
\mathbb{D}_{r}$, which implies $m_{+}\left( -z^{2}\right) =-m_{W}(z)$, \ $%
m_{-}\left( -z^{2}\right) =m_{W}(-z)$. Lemma\ref{l8} implies $-r^{2}\leq\inf$
\textrm{sp}$L_{+}$.\medskip
\end{proof}

In the above proof the normality of $\left\{ m_{e_{x}W}(z)\right\} _{x\in%
\mathbb{R}}$ was crucial. It should be remarked that the converse statement
holds. Namely, for $W\in Gr_{+}^{(2)}$ assume $m_{W}$ is connected with the
Weyl functions $m_{\pm}$ as in (\ref{56}). Then, Lemma\ref{l11} implies that
there exists a measure $\sigma_{y}$ on $\left[ -\sqrt{2}r,\sqrt{2}r\right] $
such that%
\begin{equation}
m_{e_{y}W}(z)=\sqrt{z^{2}+r^{2}}+\int_{-\sqrt{2}r}^{\sqrt{2}r}\dfrac {%
\sigma_{y}\left( d\xi\right) }{\xi-\sqrt{z^{2}+r^{2}}}\text{,}  \label{58}
\end{equation}
hence%
\begin{equation*}
m_{e_{y}W}(x)-m_{e_{y}W}(-x)=2\sqrt{x^{2}+r^{2}}\left( 1+\int_{-\sqrt{2}r}^{%
\sqrt{2}r}\dfrac{1}{\xi^{2}-\left( x^{2}+r^{2}\right) }\sigma_{y}\left(
d\xi\right) \right) \text{.}
\end{equation*}
On the other hand, Corollary\ref{c4} shows for $x>r$%
\begin{equation*}
m_{e_{y}W}(x)-m_{e_{y}W}(-x)>0\text{,}
\end{equation*}
which means%
\begin{equation*}
\int_{-\sqrt{2}r}^{\sqrt{2}r}\dfrac{1}{\left( x^{2}+r^{2}\right) -\xi^{2}}%
\sigma_{y}\left( d\xi\right) <1\text{ \ for any }x>r\text{.}
\end{equation*}
Consequently, letting $x\rightarrow r$, we have%
\begin{equation*}
\int_{-\sqrt{2}r}^{\sqrt{2}r}\dfrac{1}{2r^{2}-\xi^{2}}\sigma_{y}\left(
d\xi\right) \leq1\text{ \ for any }y\in\mathbb{R}\text{,}
\end{equation*}
which implies the normality of $\left\{ m_{e_{x}W}(z)\right\} _{x\in \mathbb{%
R}}$.

Proposition\ref{p7} asserts that the $m$-function is directly related to the
Weyl functions. Set%
\begin{equation*}
\mathcal{H}=\left\{ 
\begin{array}{c}
m\text{; \ }m\text{ is holomorphic on }\mathbb{C}\backslash\left( \mathbb{R}%
\cup i\mathbb{R}\right) \text{ satisfying} \\ 
m\left( z\right) =\overline{m\left( \overline{z}\right) }\text{ and }\func{Im%
}m(z)>0\text{ on }\mathbb{C}_{+}\backslash i\mathbb{R}%
\end{array}
\right\} \text{.}
\end{equation*}
We introduce a subclass $\mathcal{M}_{r}$ of $\mathcal{H}$ in view of the
property of Proposition\ref{p7}. Denote by $\mathcal{M}_{r}$ the set of all
functions $m$ satisfying the following conditions: Let $I_{r}=\left[ -r,%
\text{ }r\right] $ for $r>0$.%
\begin{equation}
\left\{ 
\begin{tabular}{l}
$\text{(i) \ \ }m\in\mathcal{H}$. \\ 
$\text{(ii)}\ \ m\text{ is holomorphic on }\mathbb{C}\backslash\left(
I_{r}\cup iI_{r}\right) $,$\text{ continuous on }\partial\mathbb{D}_{r}$ \\ 
$\ \ \ \ \text{and satisfies }m(r)>m(-r)$. \\ 
(iii) $m\text{ has a pole at }\infty$ $\text{of a form \ }m(z)=z+O\left(
z^{-1}\right) \text{.}$%
\end{tabular}
\right.  \label{37}
\end{equation}

The next goal is to show the converse statement. For that purpose recall
transformation%
\begin{equation*}
\left( d_{\zeta}m\right) \left( z\right) =\dfrac{z^{2}-\zeta^{2}}{m\left(
z\right) -m\left( \zeta\right) }-m\left( \zeta\right) \text{,}
\end{equation*}
and define%
\begin{equation}
\left( D_{\zeta}m\right) \left( z\right) =\dfrac{z-\zeta}{m\left( z\right)
-m\left( \zeta\right) }-m\left( \zeta\right) \text{.}  \label{38}
\end{equation}
Then, without difficulty one can show $D_{\zeta_{1}}D_{\zeta_{2}}=D_{\zeta
_{2}}D_{\zeta_{1}}$, $d_{\zeta_{1}}d_{\zeta_{2}}=d_{\zeta_{2}}d_{\zeta_{1}}$.

\begin{lemma}
\label{l9}It holds that $d_{\overline{\zeta}}d_{\zeta}m\in\mathcal{H}$ for $%
\zeta\in\mathbb{C}\backslash\left( \mathbb{R}\cup i\mathbb{R}\right) $ and $%
m\in\mathcal{H}$.
\end{lemma}

\begin{proof}
For $m\in\mathcal{H}$ define%
\begin{equation}
m_{+}(z)=m\left( \sqrt{z}\right) \text{,}  \label{39}
\end{equation}
where $\sqrt{z}$ is defined on $\mathbb{C}\backslash\mathbb{R}_{-}$ so that $%
\sqrt{1}=1$. Then, $m_{+}$ turns out to be an irrational Herglotz function,
and Lemma\ref{l10} in the Appendix implies $D_{\overline{\zeta}%
}D_{\zeta}m_{+}$ is of Herglotz, hence for $\zeta$, $z$ satisfying $%
\zeta^{2}\in\mathbb{C}\backslash\mathbb{R}$ and $\func{Re}z>0$, $\func{Im}%
z>0 $ we see 
\begin{equation*}
\func{Im}\left( d_{\overline{\zeta}}d_{\zeta}m\right) \left( z\right) =\func{%
Im}\left( D_{\overline{\zeta}^{2}}D_{\zeta^{2}}m_{+}\right) \left(
z^{2}\right) >0\text{,}
\end{equation*}
because $z^{2}\in\mathbb{C}_{+}$. To obtain the result for $z$ satisfying $%
\func{Re}z<0$, $\func{Im}z>0$ we define%
\begin{equation*}
m_{-}(z)=-m\left( -\sqrt{z}\right) \text{.}
\end{equation*}
Then, $m_{-}$ is again of Herglotz, hence%
\begin{equation*}
\func{Im}\left( d_{\overline{\zeta}}d_{\zeta}m\right) \left( z\right) =-%
\func{Im}\left( D_{\overline{\zeta}^{2}}D_{\zeta^{2}}m_{-}\right) \left(
z^{2}\right) >0
\end{equation*}
due to $z^{2}\in\mathbb{C}_{-}$ and $m(z)=-m_{-}\left( z^{2}\right) $, which
completes the proof.\medskip
\end{proof}

\begin{proposition}
\label{p11}Let $m\in\mathcal{M}_{r}$ and $s>r$. Set%
\begin{equation*}
W_{m}=\left\{ \varphi\left( z^{2}\right) +\psi\left( z^{2}\right) m(z)\text{%
; \ }\varphi,\text{ }\psi\in H_{+}\left( \partial\mathbb{D}_{s}\right)
\right\} \text{.}
\end{equation*}
Then, $W_{m}\in Gr_{+}^{\left( 2\right) }\left( \mathbb{D}_{s}\right) $ and $%
m_{W}=m$, $\tau_{W}\left( g\right) =\tau_{m}(g)$ hold. Moreover, $m_{gW}\in%
\mathcal{M}_{s}$ is valid for any $g\in\mathit{\Gamma}_{\func{real}}\left( 
\mathbb{D}_{s}\right) $.
\end{proposition}

\begin{proof}
Since%
\begin{equation*}
\boldsymbol{W}_{m}=\left\{ \boldsymbol{f}\left( z\right) =\left( 
\begin{array}{cc}
1 & m_{e}\left( z\right) \\ 
0 & m_{o}\left( z\right)%
\end{array}
\right) \left( 
\begin{array}{c}
\varphi\left( z\right) \\ 
\psi\left( z\right)%
\end{array}
\right) \text{; \ }\varphi,\text{ }\psi\in H_{+}\left( \partial \mathbb{D}%
_{s}\right) \right\} \text{,}
\end{equation*}
Applying Lemma\ref{l12} to $a(z)=m_{o}\left( z\right) $ on $\partial \mathbb{%
D}_{s}$ shows $T\left( m_{o}\right) $ is invertible on $H_{+}$, hence $%
\boldsymbol{W}_{m}\in Gr^{\left( 2\right) }\left( \mathbb{D}_{s}\right) $,
and $W_{m}\in Gr^{\left( 2\right) }\left( \mathbb{D}_{s}\right) $. In this
case%
\begin{equation*}
\varphi_{W_{m}}\left( z\right) =0\text{, \ }\psi_{W_{m}}\left( z\right)
=m(z)-z
\end{equation*}
hold, hence $m_{W_{m}}=m$, $\tau_{W_{m}}\left( g\right) =\tau_{m}\left(
g\right) $ follow. Then, the rest of the proof is to show the property $%
\tau_{W_{m}}\left( g\right) \geq0$ for $g=\prod_{k=1}^{n}q_{\zeta_{k}}q_{%
\overline{\zeta}_{k}}$ with $\zeta_{k}\in\left\{ \left\vert z\right\vert
>s\right\} $ and $\func{Im}\zeta_{k}\neq0$. We show $\tau_{W_{m}}\left(
g\right) >0$ by induction. For $n=1$ (\ref{51}) implies%
\begin{equation*}
\tau_{W_{m}}\left( q_{\zeta}q_{\overline{\zeta}}\right) =\left\vert
1+\varphi_{W_{m}}\left( \zeta\right) \right\vert ^{2}\frac{m\left(
\zeta\right) -\overline{m\left( \zeta\right) }}{\zeta-\overline{\zeta}}=%
\frac{m\left( \zeta\right) -\overline{m\left( \zeta\right) }}{\zeta-%
\overline{\zeta}},
\end{equation*}
which is a positive quantity because of $m\in\mathcal{M}_{r}$. Suppose $%
\tau_{W}\left( g_{n-1}\right) >0$ for $g_{n-1}=\prod_{k=1}^{n-1}q_{%
\zeta_{k}}q_{\overline{\zeta}_{k}}$. Then, $g_{n-1}W_{m}\in Gr^{\left(
2\right) }\left( \mathbb{D}_{s}\right) $ and is real, hence%
\begin{equation*}
\tau_{g_{n-1}W_{m}}\left( q_{\zeta_{n}}q_{\overline{\zeta}_{n}}\right)
=\left\vert 1+\varphi_{g_{n-1}W_{m}}\left( \zeta_{n}\right) \right\vert ^{2}%
\frac{m_{g_{n-1}W_{m}}\left( \zeta_{n}\right) -\overline{m_{g_{n-1}W_{m}}%
\left( \zeta_{n}\right) }}{\zeta_{n}-\overline{\zeta}_{n}}\text{.}
\end{equation*}
Since (\ref{50}) implies $m_{q_{\zeta}W_{m}}(z)=\left( d_{\zeta}m\right)
\left( z\right) $, we have%
\begin{equation}
m_{g_{n-1}W_{m}}\left( z\right) =\left( d_{\overline{\zeta}%
_{n-1}}d_{\zeta_{n-1}}\cdots d_{\overline{\zeta}_{1}}d_{\zeta_{1}}m\right)
\left( z\right) \text{.}  \label{40}
\end{equation}
Therefore, Lemma\ref{l9} implies $m_{g_{n-1}W_{m}}\in\mathcal{H}$ due to $%
m\in\mathcal{H}$, hence $m_{g_{n-1}W_{m}}\left( z\right) \in\mathbb{C}_{+}$
for any $z\in\mathbb{C}_{+}\backslash\left( \mathbb{R}\cup i\mathbb{R}%
\right) $. Consequently, we see $\tau_{g_{n-1}W_{m}}\left( q_{\zeta_{n}}q_{%
\overline{\zeta}_{n}}\right) \geq0$ for any $\zeta_{n}\in\mathbb{C}%
_{+}\backslash\left( I_{s}\cup iI_{s}\right) $, and $\tau_{g_{n-1}W_{m}}%
\left( q_{\zeta_{n}}q_{\overline{\zeta}_{n}}\right) >0$ there from (ii) of
Lemma\ref{l6}. This completes the induction and we have%
\begin{equation*}
\tau_{W_{m}}\left( g\right) =\tau_{W_{m}}\left( g_{n-1}\right)
\tau_{g_{n-1}W_{m}}\left( q_{\zeta_{n}}q_{\overline{\zeta}_{n}}\right) >0%
\text{.}
\end{equation*}
The last statement is easily verified by starting from $g=g_{n}$ and noting $%
m_{g_{n}W_{m}}\in\mathcal{H}$. The property (ii) in (\ref{37}) follows from
Corollary\ref{c4}, since $gW_{m}\in Gr_{+}^{\left( 2\right) }$.
\end{proof}

\subsection{Reflectionless property of underlying potentials}

Reflectionless property was originally introduced for decaying potentials
with vanishing reflection coefficients. However, for our purpose it is
better to define this property for more general potentials.

For a real valued $q\in L_{loc}^{1}\left( \mathbb{R}\right) $ with $%
\pm\infty $ boundaries of limit point type let $m_{\pm}$ be the Weyl
functions. Let $F$ be a Borel set in $\mathbb{R}$ with positive Lebesgue
measure. Then, $q$ is called \emph{reflectionless} on $F$ if%
\begin{equation*}
m_{+}\left( \xi+i0\right) =-\overline{m_{-}\left( \xi+i0\right) }\text{ \ \
\ \ for a.e. \ }\xi\in F
\end{equation*}
holds. It can be shown without difficulty that $\overline{F}\subset$ \textrm{%
sp} $L_{q}$, where $L_{q}$ is the Schr\"{o}dinger operator with potential $q$%
. In particular, any periodic potential is reflectionless on the spectrum.

Define%
\begin{equation}
m(z)=\left\{ 
\begin{array}{c}
-m_{+}\left( -z^{2}\right) \text{ \ for }\func{Re}z>0 \\ 
m_{-}\left( -z^{2}\right) \text{ \ \ \ for }\func{Re}z<0%
\end{array}
\right. \text{.}  \label{41}
\end{equation}
If $F\in\mathcal{B}\left( \mathbb{R}_{+}\right) $, then $q$ is
reflectionless on $F$ if and only if%
\begin{equation}
m\left( i\xi+0\right) =m\left( i\xi-0\right) \text{ \ \ for a.e. }\xi \in%
\sqrt{F}  \label{42}
\end{equation}
holds.

Define%
\begin{equation*}
\mathcal{Q}_{r}=\left\{ q=q_{W}\text{; \ \ }W\in Gr_{+}^{\left( 2\right)
}\left( \mathbb{D}_{r}\right) \text{ and }m_{W}\in\mathcal{M}_{r}\right\} 
\text{.}
\end{equation*}

\begin{proposition}
\label{p8}If $q\in\mathcal{Q}_{r}$, then, $q$ is reflectionless on $\left(
r^{2},\text{ }\infty\right) $ and \textrm{sp} $L_{q}\subset\lbrack-r^{2},$ $%
\infty)$. Conversely, if $q$ is reflectionless on $\left( r^{2},\text{ }%
\infty\right) $ and \textrm{sp} $L_{q}\subset\lbrack-r^{2},$ $\infty)$, then 
$q\in\mathcal{Q}_{r}$.
\end{proposition}

\begin{proof}
The first assertion follows from Lemma\ref{l8} and Proposition\ref{p7}. Due
to Proposition\ref{p11} it is sufficient to show $m\in\mathcal{M}_{r}$ for
the proof of the second assertion. Since (\ref{42}) for $\left( r,\text{ }%
\infty\right) $ implies $m$ is holomorphic outside of $\left[ -r\text{, }r%
\right] \cup i\left[ -r\text{, }r\right] $, and the other properties of $%
\mathcal{M}_{r}$ are clearly satisfied by $m$, we have $m\in\mathcal{M}_{r}$.
\end{proof}

\subsection{Proof of Theorem 2}

Now we construct the KdV flow. Set%
\begin{equation*}
\left\{ 
\begin{array}{l}
\mathcal{M}_{\infty}=\bigcup\limits_{r>0}\mathcal{M}_{r}\smallskip \\ 
\mathit{\Gamma}_{\func{real}}^{\infty}=\left\{ g=e^{h}\text{; \ \ }h\text{
is an entire function with }\overline{h}=h\text{.}\right\}%
\end{array}
\right. \text{.}
\end{equation*}
Then, Proposition\ref{p8} shows%
\begin{equation*}
\mathcal{Q}_{\infty}=\dbigcup _{r>0}\mathcal{Q}_{r}\text{,}
\end{equation*}
where the definition of $\mathcal{Q}_{\infty}$ is given (\ref{59}). On the
other hand, Proposition\ref{p11} shows $\mathcal{M}_{\infty}$ corresponds to 
$\mathcal{Q}_{\infty}$ through (\ref{41}) one to one. For $q\in\mathcal{Q}%
_{\infty}$ define $m\in\mathcal{M}_{\infty}$ by (\ref{41}). Then, $gW_{m}\in
Gr_{\mathcal{+}}^{\left( 2\right) }$ for $g\in\mathit{\Gamma}_{\func{real}%
}^{\infty}$ holds for some $r>0$ due to Proposition\ref{p11}. Then, one can
define $\widetilde{q}(x)=-2\partial _{x}^{2}\log\tau_{gW_{m}}\left(
e_{x}\right) \in\mathcal{Q}_{\infty}$. The property of $\tau_{W}$ and
Proposition\ref{p11} show%
\begin{equation*}
\tau_{gW_{m}}\left( e_{x}\right) =\tau_{W_{m}}\left( ge_{x}\right)
/\tau_{W_{m}}\left( g\right) =\tau_{m}\left( ge_{x}\right) /\tau _{m}\left(
g\right) \text{,}
\end{equation*}
hence%
\begin{equation*}
\widetilde{q}(x)=-2\partial_{x}^{2}\log\tau_{m}\left( ge_{x}\right) \text{,}
\end{equation*}
which is denoted by $\left( K(g)q\right) \left( x\right) $. The flow
property of $K(g)$ is verified as follows. Since the potential $K\left(
g_{2}\right) q$ is associated with $g_{2}W_{m}\in Gr_{+}^{\left( 2\right) }$%
, we see%
\begin{align*}
\left( K\left( g_{1}g_{2}\right) q\right) \left( x\right) &
=-2\partial_{x}^{2}\log\tau_{g_{1}g_{2}W_{m}}\left( e_{x}\right) \\
& =-2\partial_{x}^{2}\log\tau_{g_{1}\left( g_{2}W_{m}\right) }\left(
e_{x}\right) =\left( K(g_{1})K\left( g_{2}\right) q\right) \left( x\right) 
\text{.}
\end{align*}

\section{Proof of Theorem 3}

In this section we give a more concrete representation of $%
\tau_{m}(g)=\tau_{W_{m}}(g)$.

For $m\in\mathcal{M}_{r}$ set%
\begin{equation*}
\mathit{\Pi}_{m}=\left( 
\begin{array}{cc}
1 & m_{e} \\ 
0 & m_{o}%
\end{array}
\right) \text{, \ \ \ \ \ }\boldsymbol{W}_{m}=\mathit{\Pi}_{m}\boldsymbol{H}%
_{+}\text{.}
\end{equation*}
Then, $\boldsymbol{W}_{m}\in Gr^{\left( 2\right) }\left( \mathbb{D}%
_{s}\right) $ for $s>r$ and its characteristic matrix is $\mathit{\Pi}_{m}$.
Since $g^{-1}\mathfrak{p}_{+}gA_{W_{m}}$ is unitarily equivalent to $G^{-1}%
\mathfrak{p}_{+}GA_{\boldsymbol{W}_{m}}$, where%
\begin{equation*}
\left\{ 
\begin{array}{l}
G=\left( 
\begin{array}{cc}
g_{e}(z) & zg_{o}(z) \\ 
g_{o}(z) & g_{e}(z)%
\end{array}
\right) \smallskip \\ 
A_{\boldsymbol{W}_{m}}\boldsymbol{u}=\mathit{\Pi}_{m}\boldsymbol{T}\left( 
\mathit{\Pi}_{m}\right) ^{-1}\boldsymbol{u}-\boldsymbol{u}%
\end{array}
\right. \text{,}
\end{equation*}
its $\tau$-function is%
\begin{equation*}
\tau_{W_{m}}(g)=\det\left( I+G^{-1}\mathfrak{p}_{+}GA_{\boldsymbol{W}%
_{m}}\right) \text{.}
\end{equation*}
In this case everything is discussed in the Hilbert space $H=L^{2}\left(
\partial\mathbb{D}_{s}\right) $ and $\boldsymbol{H}=H\times H$. As we have
seen in the first section, $\boldsymbol{T}\left( \mathit{\Pi}_{m}\right) $
is invertible if and only if so is $T\left( m_{o}\right) $. The present $%
m_{o}$ satisfies the condition of Lemma\ref{l12}, hence we have%
\begin{equation*}
\boldsymbol{T}\left( \mathit{\Pi}_{m}\right) ^{-1}=\left( 
\begin{array}{cc}
I & -T\left( m_{e}\right) T\left( m_{o}^{-1}\right) \\ 
0 & T\left( m_{o}^{-1}\right)%
\end{array}
\right) \text{.}
\end{equation*}
Therefore, for $\boldsymbol{u}=u_{1}\boldsymbol{e}_{1}+u_{2}\boldsymbol{e}%
_{2}\in\boldsymbol{H}_{+}$%
\begin{align*}
& A_{\boldsymbol{W}_{m}}\boldsymbol{u} \\
& =\left( u_{1}-T\left( m_{e}\right) T\left( m_{o}^{-1}\right)
u_{2}+m_{e}T\left( m_{o}^{-1}\right) u_{2}\right) \boldsymbol{e}%
_{1}+m_{o}T\left( m_{o}^{-1}\right) u_{2}\boldsymbol{e}_{2}-\boldsymbol{u} \\
& =\mathfrak{p}_{-}m_{e}T\left( m_{o}^{-1}\right) u_{2}\boldsymbol{e}_{1}+%
\mathfrak{p}_{-}m_{o}T\left( m_{o}^{-1}\right) u_{2}\boldsymbol{e}_{2},
\end{align*}
thus%
\begin{align}
& G^{-1}\mathfrak{p}_{+}GA_{\boldsymbol{W}_{m}}\boldsymbol{u}  \notag \\
& =G^{-1}\mathfrak{p}_{+}G\mathfrak{p}_{-}m_{e}T\left( m_{o}^{-1}\right)
u_{2}\boldsymbol{e}_{1}+G^{-1}\mathfrak{p}_{+}G\mathfrak{p}_{-}m_{o}T\left(
m_{o}^{-1}\right) u_{2}\boldsymbol{e}_{2}\text{.}  \label{44}
\end{align}
Then, denoting by $\mathfrak{\pi}_{1}$, $\mathfrak{\pi}_{2}$ the projections%
\begin{equation*}
\mathfrak{\pi}_{1}\boldsymbol{u}=\left( \boldsymbol{u}\cdot\boldsymbol{e}%
_{1}\right) \boldsymbol{e}_{1}\text{, \ \ }\mathfrak{\pi}_{2}\boldsymbol{u}%
=\left( \boldsymbol{u}\cdot\boldsymbol{e}_{2}\right) \boldsymbol{e}_{2}\text{%
,}
\end{equation*}
from (\ref{44}) we see $G^{-1}\mathfrak{p}_{+}GA_{\boldsymbol{W}_{m}}%
\mathfrak{\pi}_{1}=0$ and%
\begin{equation*}
I+G^{-1}\mathfrak{p}_{+}GA_{\boldsymbol{W}_{m}}=\left( 
\begin{array}{cc}
I_{1} & \mathfrak{\pi}_{1}G^{-1}\mathfrak{p}_{+}GA_{\boldsymbol{W}}\mathfrak{%
\pi}_{2} \\ 
0 & I_{2}+\mathfrak{\pi}_{2}G^{-1}\mathfrak{p}_{+}GA_{\boldsymbol{W}}%
\mathfrak{\pi}_{2}%
\end{array}
\right) \text{,}
\end{equation*}
hence%
\begin{equation*}
\tau_{W_{m}}(g)=\det\left( I_{2}+\mathfrak{\pi}_{2}G^{-1}\mathfrak{p}_{+}GA_{%
\boldsymbol{W}}\mathfrak{\pi}_{2}\right) \text{.}
\end{equation*}
On the other hand, (\ref{44}) implies also%
\begin{align*}
& \mathfrak{\pi}_{2}G^{-1}\mathfrak{p}_{+}GA_{\boldsymbol{W}}\mathfrak{\pi }%
_{2} \\
& =\left( \left( \widehat{g}_{o}\mathfrak{p}_{+}g_{e}+\widehat{g}_{e}%
\mathfrak{p}_{+}g_{o}\right) \mathfrak{p}_{-}m_{e}+\left( \widehat {g}_{o}%
\mathfrak{p}_{+}zg_{o}+\widehat{g}_{e}\mathfrak{p}_{+}g_{e}\right) \mathfrak{%
p}_{-}m_{o}\right) T\left( m_{o}^{-1}\right) \\
& =\left( \left( \widehat{g}_{o}\mathfrak{p}_{+}g_{e}+\widehat{g}_{e}%
\mathfrak{p}_{+}g_{o}\right) m_{e}+\left( \widehat{g}_{o}\mathfrak{p}%
_{+}zg_{o}+\widehat{g}_{e}\mathfrak{p}_{+}g_{e}\right) m_{o}\right) T\left(
m_{o}^{-1}\right) \\
& -\left( \left( \widehat{g}_{o}\mathfrak{p}_{+}g_{e}+\widehat{g}_{e}%
\mathfrak{p}_{+}g_{o}\right) T\left( m_{e}\right) +\left( \widehat {g}_{o}%
\mathfrak{p}_{+}zg_{o}+\widehat{g}_{e}\mathfrak{p}_{+}g_{e}\right) T\left(
m_{o}\right) \right) T\left( m_{o}^{-1}\right) \\
& =\left( \left( \widehat{g}_{o}\mathfrak{p}_{+}g_{e}+\widehat{g}_{e}%
\mathfrak{p}_{+}g_{o}\right) m_{e}+\left( \widehat{g}_{o}\mathfrak{p}%
_{+}zg_{o}+\widehat{g}_{e}\mathfrak{p}_{+}g_{e}\right) m_{o}\right) T\left(
m_{o}^{-1}\right) -I \\
& =\left( \widehat{g}_{o}T\left( \left( gm\right) _{e}\right) +\widehat{g}%
_{e}T\left( \left( gm\right) _{o}\right) -T\left( m_{o}\right) \right)
T\left( m_{o}^{-1}\right)
\end{align*}
with $\widehat{g}=g^{-1}$. Therefore, we have%
\begin{equation}
\tau_{W_{m}}(g)=\det\left( I+\left( \widehat{g}_{o}T\left( \left( gm\right)
_{e}\right) +\widehat{g}_{e}T\left( \left( gm\right) _{o}\right) -T\left(
m_{o}\right) \right) T\left( m_{o}^{-1}\right) \right)  \label{45}
\end{equation}
Recall $r<s$ and $m\in\mathcal{M}_{r}$. Let $C$, $C^{\prime}$ be simple
closed curves of (\ref{f1}) in the introduction surrounding the interval $%
\left[ -r,r\right] $ and contained in $\mathbb{D}_{s}$. Then, it holds that
for $f\in H_{+}$ and $z$ located inside of $C^{\prime}$%
\begin{align*}
T\left( m_{o}^{-1}\right) f\left( z\right) & =\dfrac{1}{2\pi i}%
\int_{\left\vert \xi\right\vert =s}\dfrac{m_{o}\left( \lambda^{\prime
}\right) ^{-1}f\left( \lambda^{\prime}\right) }{\lambda^{\prime}-z}%
d\lambda^{\prime} \\
& =\dfrac{1}{2\pi i}\dint _{C^{\prime}}\dfrac{m_{o}\left(
\lambda^{\prime}\right) ^{-1}f\left( \lambda^{\prime }\right) }{%
\lambda^{\prime}-z}d\lambda^{\prime} \\
& \equiv Tf(z)\text{,}
\end{align*}
and for $z$ satisfying $\left\vert z\right\vert \leq s$%
\begin{align*}
& \left( \widehat{g}_{o}T\left( \left( gm\right) _{e}\right) +\widehat{g}%
_{e}T\left( \left( gm\right) _{o}\right) -T\left( m_{o}\right) \right)
f\left( z\right) \\
& =\dfrac{1}{2\pi i}\int_{\left\vert \lambda\right\vert =s}\dfrac{\widehat {g%
}_{o}\left( z\right) \left( gm\right) _{e}\left( \lambda\right) +\widehat{g}%
_{o}\left( z\right) \left( gm\right) _{o}\left( \lambda\right) -m_{o}\left(
\lambda\right) }{\lambda-z}f\left( \lambda\right) d\lambda \\
& =\dfrac{1}{2\pi i}\int_{C}\dfrac{\widehat{g}_{o}\left( z\right) \left(
gm\right) _{e}\left( \lambda\right) +\widehat{g}_{o}\left( z\right) \left(
gm\right) _{o}\left( \lambda\right) -m_{o}\left( \lambda\right) }{\lambda-z}%
f\left( \lambda\right) d\lambda \\
& \equiv Sf(z)\text{.}
\end{align*}
It should be noted that $S$ does not change if we replace $m$ by $\widetilde{%
m}$ in the integration on the curve $C$ defined by%
\begin{equation*}
\widetilde{m}\left( z\right) =m(z)-\delta\left( z\right) \text{,}
\end{equation*}
where $\delta_{e}$, $\delta_{o}$ are holomorphic in a simply connected
domain containing $C$. We can regard $S$ and $T$ as operators from $%
L^{2}\left( C\right) $ to $H_{+}\left( =H_{+}\left( \mathbb{D}_{s}\right)
\right) $ and from $H_{+}$ to $L^{2}\left( C\right) $ respectively. Now (\ref%
{45}) implies%
\begin{equation*}
\tau_{W_{m}}(g)=\det\left( I+ST\right) =\det\left( I+TS\right) \text{.}
\end{equation*}
The operator $TS:L^{2}\left( C\right) \rightarrow L^{2}\left( C\right) $ is%
\begin{equation*}
\left( TSu\right) \left( z\right) =\dfrac{1}{2\pi i}\dint _{C^{\prime}}%
\dfrac{m_{o}\left( \lambda^{\prime}\right) ^{-1}}{\lambda^{\prime}-z}%
d\lambda^{\prime}\dfrac{1}{2\pi i}\dint _{C}L_{g}\left(
\lambda^{\prime},\lambda\right) u\left( \lambda\right) d\lambda\text{,}
\end{equation*}
where%
\begin{equation*}
L_{g}\left( z,\lambda\right) =M_{g}\left( z,\lambda\right) -\dfrac {%
m_{o}\left( \lambda\right) }{\lambda-z}\text{.}
\end{equation*}
Refer to (\ref{57}) for the definition of $M_{g}$. Note%
\begin{align*}
& \dfrac{1}{2\pi i}\dint _{C^{\prime}}\dfrac{m_{o}\left(
\lambda^{\prime}\right) ^{-1}}{\lambda^{\prime}-z}\dfrac{m_{o}\left(
\lambda\right) }{\lambda-\lambda^{\prime}}d\lambda ^{\prime} \\
& =\dfrac{m_{o}\left( \lambda\right) }{\lambda-z}\dfrac{1}{2\pi i}\dint
_{C^{\prime}}\left( \dfrac{m_{o}\left( \lambda^{\prime}\right) ^{-1}}{%
\lambda^{\prime }-z}+\dfrac{m_{o}\left( \lambda^{\prime}\right) ^{-1}}{%
\lambda -\lambda^{\prime}}\right) d\lambda^{\prime}=0
\end{align*}
for $z$, $\lambda$ located inside of $C^{\prime}$. Consequently, we obtain%
\begin{equation*}
\left( TSu\right) \left( z\right) =\left( N_{m}(g)u\right) \left( z\right) 
\text{,}
\end{equation*}
which completes the proof of Theorem\ref{t1} in the introduction.

This formula for the $\tau$-functions makes it possible to establish a
theory in a more general framework, which will be realized in the second
paper.

\section{Appendix}

\subsection{Calculation of typical $\mathit{\Gamma}$-actions}

(\ref{20}) connects the $\tau$-function with $\varphi_{W}$, namely%
\begin{equation*}
\tau_{W}\left( q_{\zeta}\right) =1+\varphi_{W}\left( \zeta\right)
\end{equation*}
for $\zeta$ such that $\left\vert \zeta\right\vert >r$, where $q_{\zeta
}\left( z\right) =\left( 1-z/\zeta\right) ^{-1}$. To calculate $\tau_{W}$
for other $g\in\mathit{\Gamma}$ we compute $\varphi_{q_{\zeta}W}$, $%
\psi_{q_{\zeta}W}$. Let%
\begin{equation*}
H_{-}\ni\varphi_{q_{\zeta}W}\left( z\right) =\dfrac{a_{1}}{z}+\dfrac{a_{2}}{%
z^{2}}+\cdots\text{.}
\end{equation*}
Since $\left( 1-z/\zeta\right) +\left( 1-z/\zeta\right) \varphi_{q_{\zeta
}W}\left( z\right) \in W$ due to $1+\varphi_{q_{\zeta}W}\left( z\right) \in
q_{\zeta}W$, the decomposition%
\begin{align*}
\left( 1-\dfrac{z}{\zeta}\right) +\left( 1-\dfrac{z}{\zeta}\right)
\varphi_{q_{\zeta}W}\left( z\right) & =\left( 1-\dfrac{z}{\zeta}\right) -%
\dfrac{z}{\zeta}\left( \dfrac{a_{1}}{z}+\dfrac{a_{2}}{z^{2}}+\right)
+\varphi_{q_{\zeta}W}\left( z\right) \\
& =\left( 1-\dfrac{a_{1}}{\zeta}\right) -\dfrac{z}{\zeta}+H_{-}
\end{align*}
shows%
\begin{align*}
& \left( 1-\dfrac{z}{\zeta}\right) +\left( 1-\dfrac{z}{\zeta}\right)
\varphi_{q_{\zeta}W}\left( z\right) \\
& =\left( 1-\dfrac{a_{1}}{\zeta}\right) \left( 1+\varphi_{W}\left( z\right)
\right) -\dfrac{1}{\zeta}\left( z+\psi_{W}\left( z\right) \right) \text{.}
\end{align*}
We have used here the bijectivity of $\mathfrak{p}_{+}:W\rightarrow H_{+}$.
Setting $z=\zeta$, we see%
\begin{equation*}
\left( 1-\dfrac{a_{1}}{\zeta}\right) \left( 1+\varphi_{W}\left( \zeta\right)
\right) -\dfrac{1}{\zeta}\left( \zeta+\psi_{W}\left( \zeta\right) \right) =0%
\text{,}
\end{equation*}
which yields

\begin{equation*}
a_{1}=\zeta-\dfrac{\zeta+\psi_{W}\left( \zeta\right) }{1+\varphi_{W}\left(
\zeta\right) }\text{,}
\end{equation*}
hence%
\begin{align}
1+\varphi_{q_{\zeta}W}\left( z\right) & =\dfrac{\dfrac{\zeta+\psi _{W}\left(
\zeta\right) }{\zeta\left( 1+\varphi_{W}\left( \zeta\right) \right) }\left(
1+\varphi_{W}\left( z\right) \right) -\dfrac{1}{\zeta }\left(
z+\psi_{W}\left( z\right) \right) }{1-\dfrac{z}{\zeta}}  \notag \\
& =\left( 1+\varphi_{W}\left( z\right) \right) \dfrac{m_{W}\left(
\zeta\right) -m_{W}\left( z\right) }{\zeta-z}\text{.}  \label{47}
\end{align}
Similarly, for $\psi_{q_{\zeta}W}\left( z\right) =b_{1}/z+b_{2}/z^{2}+\cdots$%
\begin{equation*}
\left( 1-\dfrac{z}{\zeta}\right) z+\left( 1-\dfrac{z}{\zeta}\right)
\psi_{q_{\zeta}W}\left( z\right) =-\dfrac{b_{1}}{\zeta}\left( 1+\varphi
_{W}\left( z\right) \right) +\left( z+\psi_{W}\left( z\right) \right) -%
\dfrac{1}{\zeta}\left( z^{2}+A_{W}z^{2}\right) \text{,}
\end{equation*}
hence, setting $z=\zeta$, we have%
\begin{equation*}
-\dfrac{b_{1}}{\zeta}\left( 1+\varphi_{W}\left( \zeta\right) \right) +\left(
\zeta+\psi_{W}\left( \zeta\right) \right) -\dfrac{1}{\zeta}\left(
\zeta^{2}+\left( A_{W}z^{2}\right) \left( \zeta\right) \right) =0\text{,}
\end{equation*}
and%
\begin{equation*}
b_{1}=\zeta\dfrac{\zeta+\psi_{W}\left( \zeta\right) }{1+\varphi_{W}\left(
\zeta\right) }-\dfrac{\zeta^{2}+\left( A_{W}z^{2}\right) \left( \zeta\right) 
}{1+\varphi_{W}\left( \zeta\right) }\text{.}
\end{equation*}
To compute $A_{W}z^{2}$ we note the identity $A_{W}z^{2}\cdot=\mathfrak{p}%
_{-}z^{2}A_{W}\cdot-A_{W}\mathfrak{p}_{+}z^{2}A_{W}\cdot$, and set $\cdot=1$%
. Then, expanding $\varphi_{W}\left( z\right) =a_{1}/z+a_{2}/z^{2}+\cdots$
yields%
\begin{align*}
z^{2}+A_{W}z^{2} & =z^{2}+\mathfrak{p}_{-}z^{2}A_{W}1-A_{W}\mathfrak{p}%
_{+}z^{2}A_{W}1 \\
& =z^{2}+\mathfrak{p}_{-}z^{2}\varphi_{W}-A_{W}\mathfrak{p}%
_{+}z^{2}\varphi_{W} \\
& =\left( 1+\varphi_{W}\left( z\right) \right) \left(
z^{2}-a_{2}-a_{1}m_{W}\left( z\right) +a_{1}^{2}\right) \text{.}
\end{align*}
Therefore%
\begin{align*}
& -\dfrac{b_{1}}{\zeta}\left( 1+\varphi_{W}\left( z\right) \right) +\left(
z+\psi_{W}\left( z\right) \right) -\dfrac{1}{\zeta}\left(
z^{2}+A_{W}z^{2}\right) \\
& =\left( 1+\varphi_{W}\left( z\right) \right) \left( -\dfrac{b_{1}}{\zeta}%
+m_{W}(z)-a_{1}-\dfrac{z^{2}-a_{2}-a_{1}m_{W}\left( z\right) +a_{1}^{2}}{%
\zeta}\right) \\
& =\dfrac{1+\varphi_{W}\left( z\right) }{\zeta}\left( \left( \zeta
+a_{1}\right) \left( m_{W}(z)-m_{W}(\zeta)\right) +\zeta^{2}-z^{2}\right) 
\text{,}
\end{align*}
which shows%
\begin{align}
\dfrac{z+\psi_{q_{\zeta}W}\left( z\right) }{1+\varphi_{W}\left( z\right) } &
=\dfrac{\dfrac{1}{\zeta}\left( \left( \zeta+a_{1}\right) \left(
m_{W}(z)-m_{W}(\zeta)\right) +\zeta^{2}-z^{2}\right) }{1-\dfrac{z}{\zeta}} 
\notag \\
& =\dfrac{\left( \zeta+a_{1}\right) \left( m_{W}(z)-m_{W}(\zeta)\right) }{%
\zeta-z}+\zeta+z\text{.}  \label{48}
\end{align}
On the other hand, from (\ref{47})%
\begin{equation}
a_{1}\left( q_{\zeta}W\right) =\lim_{z\rightarrow\infty}z\varphi_{q_{\zeta
}W}\left( z\right) =-m_{W}\left( \zeta\right) +\zeta+a_{1}\text{,}
\label{49}
\end{equation}
hence, from (\ref{48}), (\ref{49})%
\begin{align}
m_{q_{\zeta}W}(z) & =\dfrac{z+\psi_{q_{\zeta}W}\left( z\right) }{%
1+\varphi_{q_{\zeta}W}\left( z\right) }+a_{1}\left( q_{\zeta}W\right)  \notag
\\
& =\dfrac{\dfrac{\left( \zeta+a_{1}\right) \left(
m_{W}(z)-m_{W}(\zeta)\right) }{\zeta-z}+\zeta+z}{\dfrac{m_{W}\left(
\zeta\right) -m_{W}\left( z\right) }{\zeta-z}}-m_{W}\left( \zeta\right)
+\zeta +a_{1}  \notag \\
& =\dfrac{z^{2}-\zeta^{2}}{m_{W}\left( z\right) -m_{W}\left( \zeta\right) }%
-m_{W}\left( \zeta\right) \text{.}  \label{50}
\end{align}
In the next step we compute $\tau_{W}\left(
q_{\zeta_{1}}q_{\zeta_{2}}\right) $. From (\ref{19}) it follows that%
\begin{equation*}
\tau_{W}\left( q_{\zeta_{1}}q_{\zeta_{2}}\right) =\tau_{W}\left(
q_{\zeta_{1}}\right) \tau_{q_{\zeta_{1}}W}\left( q_{\zeta_{2}}\right)
=\left( 1+\varphi_{W}\left( \zeta_{1}\right) \right) \left(
1+\varphi_{q_{\zeta_{1}}W}\left( \zeta_{2}\right) \right) \text{,}
\end{equation*}
hence, from (\ref{47})%
\begin{equation}
\tau_{W}\left( q_{\zeta_{1}}q_{\zeta_{2}}\right) =\left( 1+\varphi
_{W}\left( \zeta_{1}\right) \right) \left( 1+\varphi_{W}\left( \zeta
_{2}\right) \right) \dfrac{m_{W}\left( \zeta_{1}\right) -m_{W}\left(
\zeta_{2}\right) }{\zeta_{1}-\zeta_{2}}\text{.}  \label{51}
\end{equation}
In the last step we calculate $\tau_{W}\left( p_{\zeta}\right) $ where%
\begin{equation*}
p_{\zeta}\left( z\right) =1+z/\zeta=q_{-\zeta}\left( z\right) ^{-1}
\end{equation*}
with $\zeta$ such that $\left\vert \zeta\right\vert >r$. The key observation
is%
\begin{equation*}
r_{\zeta}\left( z\right) =q_{\zeta}\left( z\right) p_{\zeta}\left( z\right)
^{-1}=q_{\zeta}\left( z\right) q_{-\zeta}\left( z\right) =\left(
1-z^{2}/\zeta^{2}\right) ^{-1}\text{,}
\end{equation*}
and (iii) of Proposition\ref{p3} implies%
\begin{equation*}
\tau_{W}\left( q_{\zeta}p_{\zeta^{\prime}}\right) =\tau_{W}\left( q_{\zeta
}q_{\zeta^{\prime}}r_{\zeta^{\prime}}^{-1}\right) =\tau_{W}\left(
r_{\zeta^{\prime}}\right) ^{-1}\tau_{W}\left(
q_{\zeta}q_{\zeta^{\prime}}\right) \text{.}
\end{equation*}
We have only to apply (\ref{51}) to $r_{\zeta}=q_{\zeta}q_{-\zeta}$ to
compute $\tau_{W}\left( r_{\zeta}\right) $. Thus%
\begin{equation}
\tau_{W}\left( r_{\zeta}\right) =\left( 1+\varphi_{W}\left( \zeta\right)
\right) \left( 1+\varphi_{W}\left( -\zeta\right) \right) \dfrac {m_{W}\left(
\zeta\right) -m_{W}\left( -\zeta\right) }{2\zeta}=\det\mathit{\Pi}_{W}\left(
\zeta^{2}\right) \text{.}  \label{53}
\end{equation}
Hence, letting $\zeta\rightarrow\infty$, we have%
\begin{align*}
\tau_{W}\left( p_{\zeta^{\prime}}\right) & =\tau_{W}\left( r_{\zeta
^{\prime}}\right) ^{-1}\tau_{W}\left( q_{\zeta^{\prime}}\right) \\
& =\dfrac{2\zeta^{\prime}\left( 1+\varphi_{W}\left( \zeta^{\prime}\right)
\right) }{\left( 1+\varphi_{W}\left( \zeta^{\prime}\right) \right) \left(
1+\varphi_{W}\left( -\zeta^{\prime}\right) \right) \left( m_{W}\left(
\zeta^{\prime}\right) -m_{W}\left( -\zeta^{\prime}\right) \right) }\text{.}
\end{align*}
Similarly%
\begin{align}
& \tau_{W}\left( q_{\zeta_{1}}q_{\zeta_{2}}\cdots q_{\zeta_{n}}p_{\zeta
_{1}^{\prime}}p_{\zeta_{2}^{\prime}}\cdots p_{\zeta_{n}^{\prime}}\right) 
\notag \\
& =\tau_{W}\left( q_{\zeta_{1}}q_{\zeta_{2}}\cdots q_{\zeta_{n}}q_{\zeta
_{1}^{\prime}}q_{\zeta_{2}^{\prime}}\cdots q_{\zeta_{n}^{\prime}}r_{\zeta
_{1}^{\prime}}^{-1}r_{\zeta_{2}^{\prime}}^{-1}\cdots
r_{\zeta_{n}^{\prime}}^{-1}\right)  \notag \\
& =\tau_{W}\left(
r_{\zeta_{1}^{\prime}}^{-1}r_{\zeta_{2}^{\prime}}^{-1}\cdots
r_{\zeta_{n}^{\prime}}^{-1}\right) \tau_{W}\left(
q_{\zeta_{1}}q_{\zeta_{2}}\cdots
q_{\zeta_{n}}q_{\zeta_{1}^{\prime}}q_{\zeta_{2}^{\prime}}\cdots
q_{\zeta_{n}^{\prime}}\right)  \notag \\
& =\left( \tau_{W}\left( r_{\zeta_{1}^{\prime}}\right) \tau_{W}\left(
r_{\zeta_{2}^{\prime}}\right) \cdots\tau_{W}\left(
r_{\zeta_{n}^{\prime}}\right) \right) ^{-1}\tau_{W}\left(
q_{\zeta_{1}}q_{\zeta_{2}}\cdots
q_{\zeta_{n}}q_{\zeta_{1}^{\prime}}q_{\zeta_{2}^{\prime}}\cdots q_{\zeta
_{n}^{\prime}}\right) \text{.}  \label{52}
\end{align}
holds.

\subsection{Herglotz function}

As we have defined it in the previous section, a holomorphic function $m$ on 
$\mathbb{C}\backslash\mathbb{R}$ is called a Herglotz function if $m$
satisfies%
\begin{equation*}
m\left( z\right) =\overline{m\left( \overline{z}\right) }\text{ \ and \ }%
\dfrac{\func{Im}m(z)}{\func{Im}z}\geq0\text{ \ for any }z\in\mathbb{C}%
\backslash\mathbb{R}\text{.}
\end{equation*}
A necessary and sufficient condition for $m$ to be a Herglotz function is
that $m$ has a representation%
\begin{equation}
m\left( z\right) =\alpha+\beta z+\int_{-\infty}^{\infty}\left( \dfrac {1}{%
\xi-z}-\dfrac{\xi}{\xi^{2}+1}\right) \sigma\left( d\xi\right)  \label{54}
\end{equation}
with a real $\alpha$, non-negative $\beta$ and measure $\sigma$ on $\mathbb{R%
}$ satisfying%
\begin{equation*}
\int_{-\infty}^{\infty}\dfrac{1}{\xi^{2}+1}\sigma\left( d\xi\right) <\infty%
\text{.}
\end{equation*}

\begin{lemma}
\label{l10}Suppose $m$ is an irrational Herglotz function. Then, so is $D_{%
\overline{\zeta}}D_{\zeta}m$ for any $\zeta\in\mathbb{C}\backslash \mathbb{R}
$ (refer to (\ref{38}) for the definition of $D_{\zeta}$).
\end{lemma}

\begin{proof}
Assume $m$ has a representation (\ref{54}). Then%
\begin{equation*}
m\left( z\right) -m\left( \zeta\right) =\beta\left( z-\zeta\right) +\left(
z-\zeta\right) \int_{-\infty}^{\infty}\dfrac{\xi-\overline{\zeta}}{\xi-z}%
\sigma_{\zeta}\left( d\xi\right) \text{,}
\end{equation*}
with $\sigma_{\zeta}\left( d\xi\right) =\left\vert \xi-\zeta\right\vert
^{-2}\sigma\left( d\xi\right) $. Note $\sigma_{\zeta}\neq0$ since $m$ is
irrational. Hence%
\begin{equation*}
\frac{z-\zeta}{m\left( z\right) -m\left( \zeta\right) }=\dfrac{1}{%
\beta+\int_{-\infty}^{\infty}\dfrac{\xi-\overline{\zeta}}{\xi-z}\sigma
_{\zeta}\left( d\xi\right) }=\dfrac{1}{\gamma+\left( z-\overline{\zeta }%
\right) \int_{-\infty}^{\infty}\dfrac{\sigma_{\zeta}\left( d\xi\right) }{%
\xi-z}}\text{,}
\end{equation*}
where $\gamma=\beta+\int_{-\infty}^{\infty}\sigma_{\zeta}\left( d\xi\right) $%
. Then%
\begin{align*}
& \frac{z-\zeta}{m\left( z\right) -m\left( \zeta\right) }-\frac {\overline{%
\zeta}-\zeta}{m\left( \overline{\zeta}\right) -m\left( \zeta\right) } \\
& =\dfrac{1}{\gamma+\left( z-\overline{\zeta}\right) \int_{-\infty}^{\infty}%
\dfrac{\sigma_{\zeta}\left( d\xi\right) }{\xi-z}}-\dfrac{1}{\gamma }=\dfrac{%
-\left( z-\overline{\zeta}\right) \int_{-\infty}^{\infty}\dfrac{%
\sigma_{\zeta}\left( d\xi\right) }{\xi-z}}{\gamma\left( \gamma+\left( z-%
\overline{\zeta}\right) \int_{-\infty}^{\infty}\dfrac {\sigma_{\zeta}\left(
d\xi\right) }{\xi-z}\right) }\text{,}
\end{align*}
and%
\begin{align*}
\left( D_{\overline{\zeta}}D_{\zeta}m\right) \left( z\right) +\left(
D_{\zeta}m\right) \left( \overline{\zeta}\right) & =\dfrac{z-\overline {\zeta%
}}{\left( D_{\zeta}m\right) \left( z\right) -\left( D_{\zeta }m\right)
\left( \overline{\zeta}\right) } \\
& =\dfrac{\gamma^{2}}{-\int_{-\infty}^{\infty}\dfrac{\sigma_{\zeta}\left(
d\xi\right) }{\xi-z}}-\gamma z+\gamma\overline{\zeta}\text{.}
\end{align*}
Since the first term of the right side is a Herglotz function, let its
representation be%
\begin{equation*}
\dfrac{\gamma^{2}}{-\int_{-\infty}^{\infty}\dfrac{\sigma_{\zeta}\left(
d\xi\right) }{\xi-z}}=\alpha_{1}+\beta_{1}z+\int_{-\infty}^{\infty}\left( 
\frac{1}{\xi-z}-\dfrac{\xi}{\xi^{2}+1}\right) \sigma_{1}\left( d\xi\right) 
\text{.}
\end{equation*}
Then%
\begin{equation*}
\beta_{1}=\lim_{z\rightarrow\infty}\dfrac{\gamma^{2}z^{-1}}{-\int_{-\infty
}^{\infty}\dfrac{\sigma_{\zeta}\left( d\xi\right) }{\xi-z}}=\dfrac {%
\gamma^{2}}{\int_{-\infty}^{\infty}\sigma_{\zeta}\left( d\xi\right) }=\dfrac{%
\gamma^{2}}{\gamma-\beta}\text{,}
\end{equation*}
hence $\beta_{1}-\gamma=\gamma^{2}\left( \gamma-\beta\right)
^{-1}-\gamma=\beta\left( \gamma-\beta\right) ^{-1}\geq0$. On the other hand%
\begin{equation*}
\func{Im}\left( \gamma\overline{\zeta}-\left( D_{\zeta}m\right) \left( 
\overline{\zeta}\right) \right) =\func{Im}\left( \gamma\overline{\zeta}%
+m\left( \zeta\right) \right) =\left( \beta -\gamma\right) \func{Im}%
\zeta+\gamma\func{Im}\zeta =\beta\func{Im}\zeta\geq0\text{,}
\end{equation*}
which completes the proof.
\end{proof}

\begin{lemma}
\label{l11}For $r>0$ assume $m$ is holomorphic on $\mathbb{C}%
\backslash\left( \left[ -r,r\right] \cup i\left[ -r,r\right] \right) $ and
satisfies%
\begin{equation*}
\dfrac{\func{Im}m(z)}{\func{Im}z}>0\text{ \ on }\mathbb{C}\backslash\left( 
\mathbb{R}\cup i\mathbb{R}\right) \text{, \ }m(z)=\overline{m(\overline{z})}%
\text{,}
\end{equation*}
and%
\begin{equation*}
m(z)=z+O\left( z^{-1}\right)
\end{equation*}
as $z\rightarrow\infty$. Then, there exists a measure $\sigma$ on $\left[ -%
\sqrt{2}r,\sqrt{2}r\right] $ such that%
\begin{equation*}
m(z)=\sqrt{z^{2}+r^{2}}+\int_{-\sqrt{2}r}^{\sqrt{2}r}\dfrac{\sigma\left(
d\xi\right) }{\xi-\sqrt{z^{2}+r^{2}}}\text{.}
\end{equation*}
\end{lemma}

\begin{proof}
Set%
\begin{equation*}
\widetilde{m}\left( z\right) =m\left( \sqrt{z^{2}-r^{2}}\right) \text{.}
\end{equation*}
Then, $\widetilde{m}$ is holomorphic on $\mathbb{C}\backslash\left[ -\sqrt {2%
}r,\sqrt{2}r\right] $ and satisfies%
\begin{equation*}
\dfrac{\func{Im}\widetilde{m}\left( z\right) }{\func{Im}z}>0\text{, \ }%
\widetilde{m}\left( z\right) =\overline{\widetilde{m}\left( \overline{z}%
\right) }\text{, and }\widetilde{m}\left( z\right) =z+O(z^{-1})\text{.}
\end{equation*}
Therefore, there exists a measure $\sigma$ on $\left[ -\sqrt{2}r,\sqrt {2}r%
\right] $ such that%
\begin{equation*}
\widetilde{m}\left( z\right) =z+\int_{-\sqrt{2}r}^{\sqrt{2}r}\dfrac {%
\sigma\left( d\xi\right) }{\xi-z}\text{,}
\end{equation*}
which gives the representation for $m$.\medskip
\end{proof}

\noindent\textrm{Acknowledgement} \ \textit{The author appreciates Professor
F. Nakano for giving him valuable comments}. \textit{This research was
partly supported by JSPS KAKENHI Grant Number 26400128.}

\noindent\emph{OSAKA University \ emeritus professor}\newline
\emph{2-13-2 Sanda, Japan 669-1324}\newline
\emph{e-mail: skotani@outlook.com}

\end{document}